\documentclass[11pt]{article}

\usepackage{tikz}
\usetikzlibrary{arrows.meta,
                chains,
                positioning,
                shapes.geometric}
\usepackage{siunitx}

\tikzset{FlowChart/.style={
startstop/.style = {rectangle, draw, fill=red!30,
                    minimum width=2cm, minimum height=1cm,
                    on chain, join=by arrow},
  process/.style = {rectangle, rounded corners, draw, 
                    text width=3.5cm, minimum height=1cm, align=center,
                    on chain, join=by arrow},
    arrow/.style = {thick,-Triangle}
        }   }

\usepackage{amsmath}
\usepackage{amsfonts}
\usepackage{amssymb}
\usepackage{epsfig}
\usepackage{fancybox}

\usepackage{amsmath}
\usepackage{amsfonts}
\usepackage{amssymb}
\usepackage{dsfont}
\usepackage{mathrsfs}
\usepackage{bbm}
\usepackage{amsmath}
\usepackage{amsfonts}
\usepackage{amssymb}
\usepackage{multirow}
\usepackage{mathrsfs}

\usepackage{float}

\usepackage{makecell} 
\usepackage{boldline}
\usepackage{diagbox}
\usepackage{threeparttable}
\usepackage{adjustbox}

\usepackage{booktabs}

\usepackage[noend]{algpseudocode}

\usepackage[title]{appendix}

\usepackage[colorlinks=true]{hyperref}
\hypersetup{urlcolor=black, citecolor=blue,linkcolor=blue}

\renewcommand{\Box}{\framebox{\rule{0.3em}{0.0em}}}

\newtheorem{theorem}{Theorem}[section]
\newtheorem{theorem*}{Theorem}[subsubsection]
\newtheorem{lemma}{Lemma}[section]
\newtheorem{proposition}{Proposition}[section]
\newtheorem{example}{Example}[section]

\newtheorem{definition}{Definition}[section]
\newtheorem{assumption}{Assumption}[section]

\newtheorem{algorithm}{Algorithm}[section]

\DeclareMathOperator*{\esssup}{ess\,sup}
\DeclareMathOperator*{\essinf}{ess\,inf}

\newcommand{\setd}{{ d \kern -.15em l}}
\newcommand{\hatsetd}{ d \hat{\kern -.15em l }}
\newcommand{\dd}{\mathsf {d\kern -0.07em l}}

\newcommand{\bgeqn}{\begin{eqnarray}}
\newcommand{\edeqn}{\end{eqnarray}}
\newcommand{\bgeq}{\begin{eqnarray*}}
\newcommand{\edeq}{\end{eqnarray*}}

\newcommand{\R}{{\rm I\!R}}

\newcommand{\inmat}[1]{\mbox{\rm {#1}}}

\def\supp{{\rm supp}}

\newcommand{\be}{\begin{equation}}
\newcommand{\ee}{\end{equation}}

\def\bbe{{\Bbb{E}}} 

\renewcommand{\Box}{\hfill \rule{2.3mm}{2.3mm}}

\makeatletter
\@addtoreset{equation}{section}
\makeatother

\renewcommand{\thefootnote}{\fnsymbol{footnote}}

\title{ 
Preference Robust Generalized Shortfall Risk Measure 
Based on the Cumulative Prospect Theory When the Value Function and Weighting Functions Are Ambiguous \footnote{This project is supported by Hong Kong RGC grant 14500620.}
}

\author{
 Sainan Zhang and Huifu Xu \footnote{Department of Systems Engineering and Engineering Management,The Chinese University of Hong Kong, Shatin, N.T., Hong Kong. Email:  snzhang@se.cuhk.edu.hk, hfxu@se.cuhk.edu.hk.}
}

\date{\today}

\begin{document}

\maketitle

\renewcommand{\thefootnote}{\arabic{footnote}}

\begin{abstract}
{\color{black} 
The utility-based shortfall risk (SR) measure introduced by F\"olmer and Schied \cite{FoS02} 
has been recently extended by Mao and Cai \cite{MaC18} to cumulative prospect theory 
(CPT) based SR in order to  better capture a decision maker's utility/risk preference. 
In this paper, we consider a situation
where information on the value function and/or the weighting functions in the 
CPT based SR
is incomplete.
Instead of using partially available information to 
construct an approximate value function and weighting functions,
we propose a robust approach to define a generalized shortfall risk  which is based on a tuple of 
the worst case value/weighting functions from
ambiguity sets of plausible 
value/weighting functions identified 
via available information.
The ambiguity set may be constructed with elicited preference information (e.g.~pairwise comparison questionnaires) and subjective judgement, and the ambiguity reduces as more and more preference information is  obtained.
Under some moderate conditions, 
we demonstrate how the subproblem of the 
robust shortfall risk measure 
can be calculated by solving a linear programming problem.
To examine the performance of the proposed robust model and 
computational scheme, we 
carry out some 
numerical experiments and report 
preliminary test results. 
This work may be viewed as an extension of 
recent research of preference robust optimization 
models to the CPT based SR.
}
\end{abstract}


\section{Introduction}
\setcounter{equation}{0}

Let $\xi: (\Omega,{\cal F},\mathbb{P})\to\R$ be a random variable representing a financial loss and $l:\R\to \R$
be a convex, non-decreasing and non-constant {\em loss/disutility} function. 
The utility-based shortfall risk measure (SR) of $\xi$ is defined as
\bgeqn
\label{eq:defi-SR}
\inmat{(SR)}  \quad \inmat{SR}_{l,\lambda}^{P}(\xi):=\inf \{x\in \R:
\bbe_P[l(\xi-x)]\leq \lambda
\},
\edeqn
where the expectation is taken w.r.t.~the probability distribution of $\xi$, i.e., $P=\mathbb{P}\circ \xi^{-1}$
and $\lambda$ is the maximal tolerable expected utility loss.
The shortfall risk, denoted by $\inmat{SR}_{l,\lambda}^{P}(\xi)$, 
is the minimum cash injection to the financial position ($-\xi$)
such that the expected loss of the new position falls below the specified level $\lambda$.
The definition is introduced by F\"olmer and Schied \cite{FoS02} and has received 
wide attentions in risk management and finance 
not only because it covers some well-known risk measures such as 
value at risk (VaR), conditional value at risk (CVaR) and entropic risk measure, but also 
better captures tail losses
(by a proper choice of loss function) \cite{FoS04}. 
Moreover,  it satisfies invariance under randomization and can be used more appropriately for dynamic measurement of risks over time, see \cite{GSW08, We06}.

The SR model is essentially 
based on Von Neumann-Morgenstein's (VNM) expected utility theory \cite{VNM47} 
if we regard the loss function 
as a disutility function 
($l(y):=-u(-y)$).  
The expected utility theory
assumes that the decision maker's preference satisfies four
axiomatic properties (completeness,
transitivity, continuity, and independence).
In the literature of behavioural economics, it is well known that the theory contradicts some
experimental studies where most decision makers 
tend to underweight medium and high probabilities
and overweight low probabilities of extreme outcomes (Elsberg \cite{El1961} and Allais paradoxes \cite{Al53})
and this prompts various remedies/modifications of the VNM's expected utility model. Rank dependent expected utility (RDEU) theory (\cite{Qui82} and \cite[Chaper 5]{Qui12}) and cumulative prospect  theory 
(Tversky and Kahneman \cite{TvK92}, 
\cite[page 158]{WaT93} and \cite[Definition 1]{KSW11}) are subsequently proposed 
to address the drawbacks of the VNM's classical model.
The definition of SR in (\ref{eq:defi-SR})
inherits the drawbacks of the VNM's expected utility model and this prompts 
introduction of generalized shortfall risk measure (GSR) based on CPT:
\bgeqn
\label{eq:PR-IPC-CPT}
\inmat{(GSR-CPT)}  \quad
{\color{black}\inmat{SR}_{v,\lambda}^{w^-,w^+}(\xi)}
:=  &\inf\limits_{x\in \R} & x \nonumber \\
&{\rm s.t.}&  \mathbb
E_{w^-w^+}[v(\xi-x)]\leq \lambda,
\edeqn
where $v$ is a value function 
corresponding to disutility function 
and $w^-, w^+$ are probability weighting functions and 
\bgeqn
\label{eq:defi-bbe-wwv}
 \bbe_{w^-w^+}[v(\xi)]
 =\int_{0}^{\infty}w^+(\mathbb P(v(\xi)\geq t))dt
 -\int_{-\infty}^{0} w^-(\mathbb P(v(\xi)\leq t))dt.
\edeqn
Note that when $w^-(t)=1-w^+(1-t)$, $\forall t\in [0,1]$ and $v$ is a utility function,  
the cumulative prospect theory reduces to rank dependent 
expected utility theory (RDEU) \cite[page 19]{TuF02}. Consequently, we call 
the resulting shortfall risk calculated by  
(\ref{eq:PR-IPC-CPT}) the GSR-RDEU.
The GSR-CPT model 
is first introduced by Mao and Cai \cite{MaC18} although their actual definition is slightly more complex. Since the weighting functions distort/modify the
cumulative
distribution function of  $v(\xi)$,
$\bbe_{w^-w^+}[v(\xi)]$ is also called the
CPT-based distortion risk measure and consequently
the GSR-CPT is also known as  a generalized shortfall risk measured induced by the 
CPT {\color{black}based distortion risk measure}, see \cite{MaC18} for details. Mao and Cai \cite{MaC18} motivate 
their CPT based GSR 
by interpreting it as the regulator's requirement of 
the minimal capital 
 to cover a risk faced by a financial institution 
 or an insurance company. In that case, the regulator
may have 
different attitudes to 
the shortfall risk 
and the over-required
 capital risk and subsequently 
adopt different value functions 
and 
distortion functions to evaluate their losses and related probabilities.

The formulation of SR in
(\ref{eq:defi-SR}) is closely related 
to insurance premium calculation model. 
To see this, consider the case that $\xi: (\Omega,{\cal F},\mathbb{P})\to\R_+$ where 
$\xi$ represents an insurer's liability 
and $u:\R\to \R$ is
a non-decreasing utility function. The insurance premium problem
is to find
$x=:\Pi(\xi)$ such that $x$ solves the indifference equation
\bgeqn
\label{eq:insure-premium}
\bbe[u(c_0+x-\xi)]= u(c_0),
\edeqn
where $c_0$ is the initial wealth of the insurer and  
the indifference equation means that 
the insurer's expected utility of the overall wealth is unchanged after committing
to the insurance.
Under some moderate conditions of $u$, 
the indifference equation
is equivalent to 
\bgeqn
\label{eq:insure-premium-opt}
\Pi(\xi) =\inf \{x\in \R:
\bbe_P[-u(c_0-\xi+x)]\leq -u(c_0)\}, 
\edeqn
where the minimum is taken in the sense that the insurer is keen to minimize the premium in order to attract more customers. 
Obviously (\ref{eq:insure-premium-opt}) corresponds to 
the formulation (\ref{eq:defi-SR}) with $l(y)=-u(-y)$, 
{\color{black}$\lambda =-u(c_0)$
and $c_0=0$}, see details in our earlier paper \cite{ZhX21}.
Like the definition of SR, the insurance premium 
model (\ref{eq:insure-premium}) is also 
based on VNM's expected utility theory and hence has drawbacks as the SR model. 

Heilpern \cite{Hei03} seems to be the first to realize the issue
and propose a new insurance premium calculation model
based on the RDEU. Specifically 
he defines a RDEU generalization of zero utility principle as the solution of
$$
\bbe_{w}[u(c_0+x-\xi)]=u(c_0),
$$ 
where $\bbe_{w}[u(\xi)]:=\int_{\R} u(x)dw(F_{\xi}(x))$
based on a single weighting function $w$,
and derives the properties of the defined insurance premium, such as translation invariance, positive homogeneity,
independent additiveness and sub-additiveness.
Kaluszka and Krzeszowiec \cite{KaK12} extend 
the formulation to CPT based premium principal by lifting 
the constraints on the relationship between the weighting functions,
and demonstrate that their premium principle
enjoys the important properties such as
non-excessive loading, no unjustified risk loading and translation invariance and positive homogeneity.
Nardon and Pianca \cite{NaP19} 
introduce a premium principal based on the continuous cumulative prospect theory,
show that the principal retains translation invariance and positive homogeneity 
and discuss several applications to the pricing of insurance contracts. 
Since the new definitions can be put under the general mathematical framework of (GSR-CPT), we regard all of them as  generalized SR although their focuses might be different.

In this paper, we take on this stream of research but with 
a different focus. We 
consider a situation where
the value function and/or the weighting functions in 
(GSR-CPT) are ambiguous. As noted by Mao and Cai in the motivation of their GSR model, 
there could be a case that
the regulator of a market 
might have a number of value/weighting  functions in mind but
is short of complete information 
as to which one may capture its true risk preference.
If the number of plausible 
value/weighting  functions is finite, then one may 
use model (\ref{eq:PR-IPC-CPT}) to compute 
the generalized shortfall risk measure one by one and then 
pick up the one which suits the DM's preference best. 
However, this approach may not work when the number is infinite.
Likewise, the primary goal of 
incorporating CPT in the insurance premium calculation model is to 
better reflect the insurer's  risk preference.
However, when it comes to practical applications of the model, 
it is often difficult to determine which particular value function and/or pair of weighting functions 
will
capture the insurer's true risk preference.
In other words, 
when the (GSR-CPT) model is applied to practical decision making problems, there could  
be an ambiguity of which particular 
tuple of value functions and weighting functions at hand to be used to cater for 
the decision maker's risk preference.

There are potentially two ways to do this. One is to use partially available information such as pairwise comparisons to 
identify the values of the value functions and/or the weighting functions 
at a discrete set of points and then 
construct an approximate value function and 
a pair of approximate weighting functions by interpolation.
This kind of approach is widely used in economics for identifying an approximate utility function, see for instance
Clemen and Reilly \cite{ClR01}.
The other is to use  partially available information to construct a set of
plausible value functions/weighting functions \cite{HuS17,WaX20} 
and base the 
shortfall risk on the worst ones. 
In this paper, we proceed with the latter which is known as robust approach. 
We call it preference robust model to distinguish it from distributionally 
robust models in
the literature of shortfall risk measures and distortion measures
(see e.g.~Guo and Xu \cite{GuX18}, Cai et al.~\cite{CLM20}, and Pesenti et al.~\cite{PWW20}).
Delage et al.~\cite{DGX18} seem to be the first to consider a preference robust version of 
SR  (\ref{eq:defi-SR}). By assuming that the decision maker's risk preference can be characterized by 
SR or more generally convex/coherent risk measures (when $l$ is convex, SR is a convex risk measure),
they use 
DM's risk preference information 
elicited via 
certainty equivalents
and the DM's risk preference on large tail losses
to construct  an ambiguity 
set of loss functions and then use the latter to define a preference 
robust SR. This paper takes a step further to consider the (GSR-CPT) model.

To this end, we assume the decision maker's risk preference
can be represented by the GSR based on CPT which is not convex in general.
We then use the DM's 
preferences characterized by the GSR to identify a class of 
value/weighting functions and subsequently use the latter to define/calculate 
the preference robust GSR (PRGSR). 
There two ways to define PRGSR. One is to use each pair of 
value/weighting functions from the ambiguity to calculate 
a GSR-CPT via (\ref{eq:PR-IPC-CPT}) and then consider the worst 
one (with largest value),
see (\ref{eq:PRGRS-CPT-equiv}). 
The other is to replace $\bbe_{w^-w^+}[v(\xi-x)]$ in 
(\ref{eq:PR-IPC-CPT}) by the worst distorted 
expected value function in (\ref{eq:PR-IPC-CPT}) and 
define the corresponding shortfall risk value (the optimal value) as the PRGSR,
see (\ref{eq:PRGRS-CPT}).
We adopt the latter and demonstrate that the two definitions 
are equivalent (Proposition~\ref{prop:PGSR-robust-equal}).

A key ingredient of the research is  construction 
of the ambiguity set of value/weighting functions.
Differing from Delage et al.~\cite{DGX18}, here
a DM's preference is described 
by both the value function and the weighting functions.
As such, 
we must infer
not only the value function (corresponding to the loss function in the SR model) 
but also the weighting functions based on 
the decision maker's preferences.
Moreover, the CPT based SR is not  
a convex risk measure in general 
when 
the value functions are strictly S-shaped.
All these pose new challenges to the 
construction of the ambiguity set and development of
computational schemes for calculating the preference 
robust shortfall risk measure, this paper aims to address them.
We propose to use certainty equivalent and pairwise comparison approaches 
to infer the value function for a fixed pair of 
weighting functions and then for a fixed value function, 
 we construct a ball of weighting functions
 centered at a nominal weighting function under some  
 pseudo-metric.
 The latter is based on the fact that
 in the literature of behavioural economics, 
 weighting functions of decision makers are mostly similar \cite{TvK92}.

Another important component of this work is to 
 develop computational methods which can be effectively
 used to calculate the proposed preference robust shortfall risk measure. 
To this end, we derive 
reformulations of the robust shortfall risk measure
when the ambiguity set of value functions 
and the ambiguity set of weighting functions
take specific forms. 
Since the choice of worst case 
value function and worst case weighting functions 
inevitably results in
an optimization problem with bilinear objective function, we reformulate the latter as a linear program by replacing the bilinear terms with new variables.
To examine the effectiveness of the ways proposed for constructing 
the ambiguity sets and the subsequent computational schemes, we carry out some numerical tests on an artificially constructed academic example. The preliminary results show that  convergence
of the worst case value functions and weighting functions to their true counterparts 
as the ambiguity reduces (with more preference 
information being elicited or the radius of 
the ball of weighting functions driven to zero).

The rest of the paper is organized as follows.
Section \ref{Sec:robust-GSR-CPT}
introduces the preference robust GSR model 
and discusses the equivalent definition and 
the main properties of the robust GSR.
Section \ref{sec3:ambiguity_VW}
discusses construction of the ambiguity sets of both the
value functions and the weighting functions. 
Section \ref{sec:Computational-schemes}
presents details 
of computational schemes for calculating the 
preference robust shortfall risk measure 
via reformulations. 
Finally, Section \ref{sec5:case_study} 
reports 
numerical test results.

Throughout the paper,
we use $\mathscr V$ 
to denote the set of all value functions $v:\R\rightarrow \R$,
which are strictly increasing with $v(0)=0$.
To facilitate properties of the value function over
the negative half line and positive half, we write 
$v^-(t)$ for $v(t)$ when
$t<0$,
and $v^+(t)$ for the value function when
$t>0$.
Furthermore, we use 
$\mathscr W$ to denote
the set of all weighting/distortion functions $w:[0,1]\rightarrow [0,1]$,
which are strictly increasing 
with $w(0)=0,w(1)=1$,
${\cal W}_{inS}$
the  subset of weighting functions
which are inverse $S$-Shaped, that is,
concave in $(0,p^*)$ 
and convex in $(p^*,1)$ with some $p^*\in (0,1)$. 
In some cases, we need to discuss the derivatives of the weighting functions and consequently use
$\psi^-(t)$
and $\psi^+(t)$
to denote the derivatives of $w^-$ and $w^+$ respectively. 
Finally, by convention, we use
${\cal L}^0(\R)$ to represent 
the space of measurable random variables over $\R$.

In this paper, we use a number of acronyms 
for both new notions and mathematical models 
to facilitate exposition/reading. Below is a list.
\begin{itemize}
    \item SR -- shortfall risk measure
    \item CPT -- cumulative prospect theory
    \item GSR -- generalized shortfall risk measure
    \item GSR-CPT -- generalized shortfall risk measure based on cumulative prospect theory
    \item GSR-RDEU -- generalized shortfall risk measure based on rank dependent expected utility theory
    \item PRGSR -- preference robust generalized shortfall risk measure
    \item PRGSR-CPT -- preference robust generalized shortfall risk measure based on cumulative prospect theory
    \item Worst-GSR-CPT -- worst generalized shortfall risk measure based on cumulative prospect theory
    \item PRGSR-CPT-S -- specific preference robust generalized shortfall risk measure based on cumulative prospect theory
    
    \item PRGSR-CPT-S-Dis -- specific preference robust generalized shortfall risk measure based on cumulative prospect theory for discrete random variable
\end{itemize}
When we refer to the mathematical formulation of a particular concept, we use a round parenthesis, e.g.,
(SR) refers to the shortfall risk model defined in (\ref{eq:defi-SR}).

\section{Preference robust GSR-CPT}
\label{Sec:robust-GSR-CPT}

We begin by a formal definition of the preference robust
generalized shortfall risk measure based on CPT (PRGSR-CPT).

\begin{definition}
[PRGSR-CPT]
Let ${\cal V}\times{\cal W}\times{\cal W}\subset {\mathscr V}\times {\mathscr W}\times {\mathscr W}$ be ambiguity sets of the value function and weighting functions.
The preference robust generalized shortfall risk measure based on cumulative prospect theory 
of random variable  
$\xi\in {\cal L}^0(\R)$ associated with ${\cal V}\times{\cal W}\times{\cal W}$  is
defined as 
\bgeqn
\label{eq:PRGRS-CPT}
\inmat{(PRGSR-CPT)} \quad \rho_{{\cal V}\times{\cal W}}(\xi)
:= 
 &\inf\limits_{x\in \R} & x \nonumber \\
&{\rm s.t.}&  \sup_{(v,w^-,w^+)\in {\cal V}\times {\cal W} \times {\cal W}} \mathbb
E_{w^-w^+}[v(\xi-x)]\leq 0.
\edeqn
\end{definition}

In this definition, the worst case distorted expected value of 
$v(\xi-x)$ is used in the constraint.
Here we make a blanket assumption 
that the robust constraint function in (\ref{eq:PRGRS-CPT}) is well-defined. Note that in
formulation (\ref{eq:PRGRS-CPT}), 
we focus on the case that $w^-,w^+$ belong to the same ambiguity set ${\cal W}$ to simplify the discussions. All of the results developed in this paper can be applied to ambiguity sets of weighting functions without this restriction.
Note also that the ambiguity sets in this definition are given. We will come back shortly to discuss about how they may be constructed based on a decision maker's preference information.

There is an alternative way  to define the PRGSR. 
For each tuple 
$(v,w^-,w^+)\in {\cal V}\times {\cal W} \times {\cal W}$, 
we calculate the CPT based shortfall risk measure:
\bgeqn 
\label{eq:rho-CPT}
\inmat{(GSR-CPT)} \quad
\rho_{(v,w^-,w^+)}(\xi):=\inf\limits_{x\in \R} \left\{ x:   \bbe_{w^-w^+}[v(\xi-x)]\leq 0 \right\}
\edeqn
and then define the PRGSR 
as the worst (largest) GSR-CPT:
\bgeqn
\label{eq:PRGRS-CPT-equiv}
\inmat{(Worst-GSR-CPT)} \quad 
 \Pi_{{\cal V} \times {\cal W}}(\xi) := \sup\limits_{(v,w^-,w^+)\in {\cal V} \times {\cal W} \times {\cal W}} \rho_{(v,w^-,w^+)}(\xi).
\edeqn 
Note that the GSR-CPT model defined in (\ref{eq:rho-CPT})
differs slightly from the one in (\ref{eq:PR-IPC-CPT}) 
in that
we set $\lambda=0$ here. This is not only for the simplification 
in the forthcoming discussions but also aligning the formulation to 
the definition of GSR in Mao and Cai \cite{MaC18}
(see Definition 2.1 there).
In the rest of the paper, (GSR-CPT) always refers to problem (\ref{eq:rho-CPT}).
Note also that under some circumstances, i.e., $\xi$ is restricted to taking non-negative values,
$w^-(p)=1-w^+(1-p)$, $\forall p\in [0,1]$
and $v(-t)=-v(t)$, $\forall t\in \R$,
$\mathbb E_{w^-w^+}[v(\xi-x)]\leq 0$ is equivalent to
$\mathbb E_{w^-w^+}[v(x-\xi)]\geq 0$. Consequently 
the (GSR-CPT) reduces to CPT based insurance premium calculation model \cite{ZhX21}.

The next proposition states that the two definitions 
(see (PRGSR-CPT) in (\ref{eq:PRGRS-CPT}) 
and (Worst-GSR-CPT) in (\ref{eq:PRGRS-CPT-equiv}))
are equivalent.

\begin{proposition}
\label{prop:PGSR-robust-equal}
Let 
$\rho_{{\cal V}\times{\cal W}}(\xi)$ and $\Pi_{{\cal V} \times {\cal W}}(\xi) $
be defined as in (\ref{eq:PRGRS-CPT}) and (\ref{eq:PRGRS-CPT-equiv}) respectively. 
Then 
\bgeqn
\rho_{ {\cal V}\times{\cal W}}(\xi) =
 \Pi_{{\cal V} \times {\cal W}}(\xi).
\label{eq:PRGSR-equivalent}
\edeqn
\end{proposition}

\begin{proof} We begin by  showing that 
the equality (\ref{eq:PRGSR-equivalent}) 
holds when one of 
the quantities is $+\infty$.
First, consider the case that
$\rho_{ {\cal V}\times{\cal W}}(\xi) =+\infty$.
Then  the feasible set of problem (\ref{eq:PRGRS-CPT}) is empty.
Assume for the sake of a contradiction that 
$\hat{x} :=\Pi_{{\cal V} \times {\cal W}}(\xi)<+\infty$.
Then $\rho_{(v,w^-,w^+)}(\xi)\leq \hat x<+\infty$ for all $(v,w^-,w^+)\in {\cal V}\times {\cal W}\times {\cal W}$.
Thus for any positive number $\epsilon$,
it follows by Lemma \ref{lemma:monotone-of-bbewv}~(i) that 
$$
\bbe_{w^-w^+}[v(\xi-\hat x-\epsilon)] \leq \bbe_{w^-w^+}[v(\xi-\rho_{(v,w^-,w^+)}(\xi)-\epsilon] \leq 0, \; \forall (v,w^-,w^+)\in {\cal V}\times {\cal W}\times {\cal W}.
$$
The latter implies $\hat x+\epsilon$ is a feasible solution to the problem (\ref{eq:PRGRS-CPT}), a contradiction.

Conversely, consider the case that
$\Pi_{{\cal V} \times {\cal W}}(\xi)=+\infty$.
Assume for a contradiction that $x^* := \rho_{ {\cal V}\times{\cal W}}(\xi)<+\infty$.
Then for any positive number $\epsilon$,
$x^*+\epsilon$ is a feasible solution to problem (\ref{eq:PRGRS-CPT}), 
that is,
$$
 \sup_{(v,w^-,w^+)\in {\cal V}\times {\cal W} \times {\cal W}} \bbe_{w^-w^+}[v(\xi-x^*-\epsilon)]\leq 0,
$$
which implies
$$
\bbe_{w^-w^+}[v(\xi-x^*-\epsilon)]\leq 0,\;\forall (v,w^-,w^+)\in {\cal V}\times {\cal W} \times {\cal W}.
$$
The latter 
guarantees that 
$\rho_{(v,w^-,w^+)}(\xi)\leq x^*+\epsilon, \forall (v,w^-,w^+)\in {\cal V}\times {\cal W} \times {\cal W}$ and hence
 $\Pi_{{\cal V} \times {\cal W}}(\xi)\leq x^*+\epsilon<+\infty$,
a contradiction.

 We now turn to discuss the case that
the optimal values of the programs in (\ref{eq:PRGRS-CPT}) and (\ref{eq:PRGRS-CPT-equiv}) are both finite.
Let $x$ be 
any feasible solution of 
problem (\ref{eq:PRGRS-CPT}).
Then
$x$ is also a feasible solution of problem (\ref{eq:rho-CPT})
for any $(v,w^-, w^+)\in {\cal V}\times {\cal W} \times {\cal W}$, which implies 
$x^*\geq \hat{x}$.

Conversely, for any fixed $(v,w^-, w^+)\in {\cal V}\times {\cal W} \times {\cal W}$,
let $x_{(v,w^-, w^+)}$ be a feasible solution to
problem (\ref{eq:rho-CPT}).
Then by definition
$
\hat{x} \geq  x_{(v,w^-, w^+)}.
$
It follows by Lemma \ref{lemma:monotone-of-bbewv}~(i)
that
$$
\mathbb{E}_{w^-w^+}[v(\xi-\hat{x})] 
\leq 
\mathbb{E}_{w^-w^+}
[v(\xi-x_{(v,w^-, w^+)})] 
\leq 0.
$$
This shows $\hat{x}$ is a feasible solution of 
problem (\ref{eq:PRGRS-CPT}).
This shows $\hat{x}\geq x^*$. 
\hfill $\Box$
\end{proof}

Next, we discuss the properties of the newly introduced PRGSR-CPT.
First, it is known that the GSR-CPT is a law invariant monetary risk measure (satisfying 
monotonicity and translation invariance), see \cite{MaC18}. Since these properties are preserved 
under ``sup'' operation, then  $\Pi_{{\cal V} \times {\cal W}}(\cdot)$ is also a law invariant 
monetary risk measure. We can therefore conclude, via Proposition~\ref{prop:PGSR-robust-equal}, 
that PRGSR-CPT is a law invariant monetary risk measure. 
Second, it might be interesting to discuss whether PRGSR-CPT  is a convex risk measure. Recall that
Mao and Cai \cite{MaC18} derive conditions under which GSR-CPT is a convex risk measure. {\color{black} The next proposition states that if we use 
the type of convex value functions and convex/concave weighting functions specified by Mao and Cai \cite{MaC18} to define the ambiguity sets 
${\cal V}$ and ${\cal W}$, then the resulting preference robust SR
is also a convex risk measure.}

\begin{proposition}
\label{Prop:Convex}
Let 
${\cal R}_{con}:=\{\rho_{(v,w^-,w^+)}:\rho_{(v,w^-,w^+)} \inmat{ is a convex risk measure}\}$,
$$
\rho_{{\cal R}_{con}}(\xi)=\sup_{\rho_{(v,w^-,w^+)}\in {\cal R}_{con}}\rho_{(v,w^-,w^+)}(\xi).
$$
Define
\bgeqn
\label{eq:PRGSR-Mao_cai}
\rho_{{\cal V}\times {\cal W}_-\times {\cal W}_+}(\xi):=\inf\limits_{x\in \R} \left\{ x :\sup_{(v,w^-,w^+)\in {\cal V}\times {\cal W}_- \times {\cal W}_+} \bbe_{w^-w^+}[v(\xi-x)]\leq 0\right\}.
\edeqn 
If ${\cal V}_{1}\subset {\mathscr V}$ is the set of all value functions satisfying that $v$ is convex on both  $\R_-$ and $\R_+$, ${\cal W}_{con}\subset {\mathscr W}$ is the set of all convex weighting functions, 
${\cal W}_{cav}\subset {\mathscr W}$ is the set of all concave weighting functions,
and
\bgeqn
{\cal V}_{con}\times {\cal W}_1\times{\cal W}_2:=\left\{(v,w^-,w^+)\in {\cal V}_1\times {\cal W}_{con}\times {\cal W}_{cav}: \frac{v'_+(0)}{v'_-(0)}\geq \sup_{p\in(0,1)} \frac{(w^-)'_-(p)}{(1-w^+(1-\cdot))'_+(p)}\right\},
\label{eq:constraint-vww}
\edeqn
where $v'_-(0)$ and $v'_+(0)$ are the left and right derivatives of value function $v$ at point $0$,
$(w^-)'_-(p)$ and $(1-w^+(1-\cdot))'_+(p)$ are the left and right derivatives of $w^-$ and $1-w^+(1-\cdot)$ at point $p$ respectively,
then 
$$
\rho_{{\cal V}_{con}\times{\cal W}_1\times {\cal W}_2}(\xi)
=\rho_{{\cal R}_{con}}(\xi).
$$
\end{proposition}

\noindent\textbf{Proof.} By Proposition \ref{prop:PGSR-robust-equal},
$$
\rho_{{\cal V}_{con}\times {\cal W}_1\times{\cal W}_2}(\xi) =
\sup_{(v,w^-,w^+)\in {\cal V}_{con}\times {\cal W}_1 \times {\cal W}_2}  \rho_{(v,w^-,w^+)}(\xi).
$$
By  \cite[Corollary 2.8]{MaC18}, 
a GSR-CPT $\rho_{(v,w^-,w^+)}(\xi)$ is a convex risk measure 
if and only if $w^-(p)$ and $1-w^+(1-p)$ are convex in $p\in [0,1]$, and $v$ is convex on both $\R_-$ and $\R_+$,
and $(v,w^-,w^+)$ satisfying the inequality in (\ref{eq:constraint-vww}), 
which means that for any $ \rho_{(v,w^-,w^+)}(\xi)\in {\cal R}_{con}$, we must have 
$v\in {\cal V}_{con}, w^-\in {\cal W }_1, w^+\in {\cal W}_2$. Thus  
$\rho_{(v,w^-,w^+)}(\xi)\leq \rho_{{\cal V}_{con} \times {\cal W}_1\times {\cal W}_2}(\xi)$ and hence
$$
\rho_{{\cal R}_{con}}(\xi) =\sup_{\rho_{(v,w^-,w^+)} \in {\cal R}_{con}}\rho_{(v,w^-,w^+)}(\xi)\leq \rho_{{\cal V}_{con} \times {\cal W}_1\times {\cal W}_2}(\xi).
$$
Conversely for any $v\in {\cal V}_{con}, w^-\in {\cal W}_1, w^+\in {\cal W}_2$,
$\rho_{(v,w^-,w^+)} \in {\cal R}_{con}$, and hence
$$
 \rho_{{\cal V}_{con}\times {\cal W}_1\times {\cal W}_2}(\xi) 
 =\sup_{(v,w^-,w^+)\in {\cal V}_{con}\times {\cal W}_1 \times {\cal W}_2}  \rho_{(v,w^-,w^+)}(\xi)
 \leq 
 \rho_{{\cal R}_{con}}(\xi).
 $$
The proof is complete.
\hfill $\Box$

We note that the PRGSR-CPT defined in (\ref{eq:PRGSR-Mao_cai}) is slightly 
different from PRGSR-CPT in (\ref{eq:PRGRS-CPT})
by allowing $w^{-}$ and $w^+$ to have different ambiguity set. This is because in the setup of the GSR-CPT models,
$w^-(p)$ and $1-w^+(1-p)$ are convex in $p\in [0,1]$ 
and consequently $w^+(p)$ is a concave function.

Mao and Cai \cite{MaC18} also derive conditions under which the GSR-CPT is a coherent risk measure. We can easily use their result for establishing coherence of 
PRGSR-CPT (\ref{eq:PRGSR-Mao_cai}).

\begin{proposition}
\label{prop:robust-RM-Pi}
Let $\rho_{{\cal V} \times {\cal W}_1 \times {\cal W}_2}(\xi)$
 be defined as in (\ref{eq:PRGSR-Mao_cai}).
Define
\bgeq
{\cal R}_{coh} &:=&\{\rho_{(v,w^-,w^+)}:\rho_{(v,w^-,w^+)} \inmat{ is a coherent risk measure}\},\\
\rho_{{\cal R}_{coh}}(\xi)&:=&\sup_{\rho_{(v,w^-,w^+)}\in {\cal R}_{coh}}\rho_{(v,w^-,w^+)}(\xi),\\
{\cal V}_{coh}&:=&\left\{v\in {\mathscr V} |\, v(x)=a_1x \inmat{ for } x\geq 0, v(x)=a_2x \inmat{ for } x<0, a_1,a_2>0 \right \}.
\edeq
If ${\cal W}_1\subset \mathscr W$ is the set of all convex weighting functions, ${\cal W}_2\subset \mathscr W$ is the set of all concave weighting functions, 
then
$$
\rho_{{\cal V}_{coh} \times {\cal W}_1\times {\cal W}_2}(\xi)
=\rho_{{\cal R}_{coh}}(\xi).
$$
\end{proposition}

\noindent\textbf{Proof.}
Observe first that since $w^-,w^+$ are increasing,
 then $a_1w^-(p)-a_2(1-w^+(1-p))$ is also increasing
 over $[0, 1]$ for all $a_1,a_2>0$.
By  \cite[Theorem 2.10]{MaC18}, 
 $\rho_{(v,w^-,w^+)}$ is coherent if and only if $w^-(p)$ and $1-w^+(1-p)$ are convex and $v\in {\cal V}_{coh}$.
Moreover, it follows by Proposition \ref{prop:PGSR-robust-equal} that
$\rho_{{\cal V}_{coh} \times {\cal W}_1 \times {\cal W}_2}(\xi)$
can be rewritten as 
$\sup_{(v,w^-,w^+)\in {\cal V}_{coh}\times {\cal W}_1 \times {\cal W}_2}  \rho_{(v,w^-,w^+)}(\xi)$.
 For any 
 $\rho_{(v,w^-,w^+)}\in {\cal R}_{coh}$, 
 we must have $(v,w^-,w^+)\in {\cal V}_{coh}\times {\cal W}_1\times {\cal W}_2$ and consequently 
 $$
 \rho_{(v,w^-,w^+)}(\xi) \leq \rho_{{\cal V}_{coh} \times {\cal W}_1 \times {\cal W}_2}(\xi)
 \quad
 \text{and}\quad  \rho_{{\cal R}_{coh}}(\xi)\leq \rho_{{\cal V}_{coh} \times {\cal W}_1 \times {\cal W}_2}(\xi).
 $$
 Conversely, for any  $(v,w^-,w^+)\in {\cal V}_{coh}\times {\cal W}_1\times {\cal W}_2$,  
 $ \rho_{(v,w^-,w^+)}  \in {\cal R}_{coh}$. 
 Thus 
 $$
 \rho_{(v,w^-,w^+)} \leq \rho_{{\cal R}_{coh}}
 \quad \text{and}\quad 
 \rho_{{\cal V}_{coh} \times {\cal W}_1 \times  {\cal W}_2}(\xi)\leq \rho_{{\cal R}_{coh}}(\xi).
 $$
The proof is complete. \hfill $\Box$

In the literature of 
prospect theory \cite{KaT79,TvK92},
the value functions are often S-shaped, which means 
that they are not convex if they are strictly S-shaped.
In this case, the GSR-CPT defined in (\ref{eq:rho-CPT}) is not convex and consequently the PRGSR-CPT defined in 
(\ref{eq:PRGSR-Mao_cai}) is not a convex risk measure.
However, if the value function $v$ is not strictly $S$-shaped, then the value function may still satisfy 
the conditions specified in the definition of ${\cal V}_{con}$ (see  Proposition \ref{Prop:Convex}). The next example is provided by Mao in a private communication.

\begin{example}
Let $w^-(p)=1-w^+(1-p), \; \forall p\in [0,1]$ 
be differentiable
and convex, let
the value function $v$ 
be convex over $\R_-$ and 
linear over $\R_+$, i.e.,
\bgeqn
v(x)=\left\{ \begin{array}{ll}
v^-(x) & \inmat{for } x<0,\\
0& \inmat{for } x=0,\\
kx & \inmat{for } x> 0,\\
\end{array}
\right.
\edeqn
where $v^-$ 
is convex over $\R_-$
and the slope $k$ satisfies that $\frac{k}{v'_-(0)}\geq  \sup_{p\in(0,1)} \frac{(w^-)'_-(p)}{(1-w^+(1-\cdot))'_+(p)}$.
Since $ \sup_{p\in(0,1)} \frac{(w^-)'_-(p)}{(1-w^+(1-\cdot))'_+(p)}=1$,
then $v$ is actually convex over the whole real line $\R$.
This shows a special S-shaped (actually convex) value function 
satisfying the conditions of ${\cal V}_{con}$.
\end{example}

The PRGSR-CPT model 
covers two important preference robust risk measure as special cases.
\begin{itemize}
\item[(1)]{\bf Preference robust utility-based shortfall risk measures.}
If the ambiguity set of weighting functions only contains single linear function, i.e., $w(p)=p$, $\forall p\in [0,1]$, 
then 
$
\Pi_{{\cal V} }(\xi)
:= \sup\limits_{v\in {\cal V}} \rho_{(v,w^-,w^+)}(\xi).
$
It follows from Proposition \ref{prop:PGSR-robust-equal} that the above risk measure is equivalent 
to the 
preference robust utility-based shortfall risk measures
\bgeq
\rho_{{\cal V}}(\xi):=\inf \limits_{x\in \R} \left\{ x: \sup_{v\in {\cal V}} \bbe[v(\xi-x)] \leq 0\right\},
\edeq
which has been fully investigated 
by
Delage et al.~\cite{DGX18}.

\item[(2)]{\bf  Preference robust distortion risk measure.}
If the ambiguity set of value functions 
is a singleton, i.e., $v(t)=t$, $\forall t\in \R$, 
then 
$$
\Pi_{{\cal W} }(\xi)
:= \sup\limits_{(w^-,w^+)\in {\cal W}\times {\cal W}} \rho_{(v,w^-,w^+)}(\xi).
$$
By Proposition \ref{prop:PGSR-robust-equal},
the above problem is equivalent to
\bgeqn
\label{eq:distortion_problem}
\rho_{{\cal W}}(\xi)
:=\inf\limits_{x\in \R} \left\{ x:  \sup_{(w^-,w^+)\in {\cal W}\times {\cal W}} 
\bbe_{w^-w^+}[\xi-x] \leq 0 \right\}.
\edeqn
Since $\bbe_{w^-w^+}[\xi]$ is essentially about distortion/modification of the CDF of $\xi$,
$\rho_{\cal W}(\xi)$ may be
viewed as a shortfall risk measure based on Yaari's dual theory of choice.
Moreover it is well-known that $\bbe_{w^-w^+}[\xi]$ is a coherent risk measure when $w^-(p)=1-w^+(1-p)$ and $w^-(p)$ is convex,
which means $\bbe_{w^-w^+}[\xi-x]=\bbe_{w^-}[\xi]-x$.
In this case, 
$\Pi_{\cal W}(\xi)$ 
coincides with 
the preference robust distortion risk measure $\sup_{w^-\in {\cal W}}\int_{\R} tdw^-(F_{\xi}(t))$ \cite{WaX21}.
\end{itemize}

Before concluding this section, we note that the reason why we define the PRGSR via (\ref{eq:PRGRS-CPT}) rather than 
(\ref{eq:PRGRS-CPT-equiv}) is that the former can be 
solved by tractable reformulations whereas the latter cannot.
However, formulation  (\ref{eq:PRGRS-CPT-equiv}) 
connects more directly to  
 decision maker's preferences. To see this, consider the
case that the a decision maker prefers prospect $A$ over prospect $B$. 
Assuming that the DM's preference can be represented by 
the CPT based GSR $\rho_{(v,w^-,w^+)}(\xi)$, then we have
\bgeqn
\rho_{(v,w^-,w^+)}(A)< \rho_{(v,w^-,w^+)}(B).
\label{eq:PRGSR-AB}
\edeqn
The inequality enables us to identify a class of value/weighting functions $v, w^-,w^+$, i.e., the ambiguity sets ${\cal V}$ and ${\cal W}$ and subsequently define PRGSR
$\Pi_{{\cal V} \times {\cal W}}(\xi)$.
To illustrate the idea, consider a simple case that
the DM's preference may be represented by three pairs 
of value/weighting functions $\{(v_i,w^-_i,w_i^+): i=1,2,3\}$.
It is found that $(v_1,w^-_1,w_1^+)$ and $(v_2,w^-_2,w_2^+)$ satisfy 
(\ref{eq:PRGSR-AB}) but $(v_3,w^-_3,w_3^+)$ does not.
That is, ${\cal V}\times {\cal W}=\{ (v_i,w^-_i,w_i^+): i=1,2\}$.
Let $\xi$ be a new prospect. Then
\bgeqn
\label{eq:PRGRS-CPT-equiv-ex}
 \Pi_{{\cal V} \times {\cal W}}(\xi)
 := \max_{i=1,2} \rho_{(v_i,w_i^-,w_i^+)}(\xi)
\edeqn 
and 
\bgeqn
\label{eq:PRGRS-CPT-ex}
\rho_{{\cal V}\times{\cal W}}(\xi)
:= \inf\limits_{x\in \R}  \left\{x: 
\sup_{i=1,2}
\mathbb E_{w_i^-w_i^+}[v_i(\xi-x)]\leq 0 \right\}.
\edeqn 
In this case, we can easily 
calculate the PRGSR via (\ref{eq:PRGRS-CPT-equiv-ex}).
In practice, however, the ambiguity sets ${\cal V} \times {\cal W}\times {\cal W}$ is not a finite set and consequently 
it is very difficult to calculate $\Pi_{{\cal V} \times {\cal W}}(\xi)$. In contrast, 
(\ref{eq:PRGRS-CPT}) can be reformulated as a linear program
when the ambiguity sets have some specific structure. We will come back to this in Section 4. From this simple example, we can see how we may use the equivalence between 
$ \Pi_{{\cal V} \times {\cal W}}(\xi)$ and 
$\rho_{{\cal V}\times{\cal W}}(\xi)$ to 
define and calculate the PRGSR: using the DM's information 
represented by GSR-CPT to construct the ambiguity sets 
${\cal V} \times {\cal W}\times {\cal W}$ and then use 
(\ref{eq:PRGRS-CPT}) to calculate the PRGSR. The 
flow chart in Figure \ref{fig:flow-chat} 
illustrates 
the procedures
for which we will 
give rise to details in the next two sections.

\begin{figure}
\centering
\begin{tikzpicture}[FlowChart,
    node distance = 5mm and 7mm,
      start chain = A going right]
\node (box1) [process]
{DM's preferences\\ 
observed and  repres-\\
ented by $\rho_{(v_i,w_i^-,w_i^+)}$};      
\node (box2) [process]
{ Construct the \\
ambiguity set \\
${\cal V} \times {\cal W}\times {\cal W}$};  

\node (box3) [process]
{Calculate PRGSR \\
by solving problem (PRGSR-CPT)};        

\draw [arrow] (box1) -- node[anchor=east] {} (box2);
\draw [arrow] (box2) -- node[anchor=south] {} (box3);
\end{tikzpicture}
 \hfil
  \caption{From DM's preferences to calculation of PRGSR}
  \label{fig:flow-chat}
\end{figure}
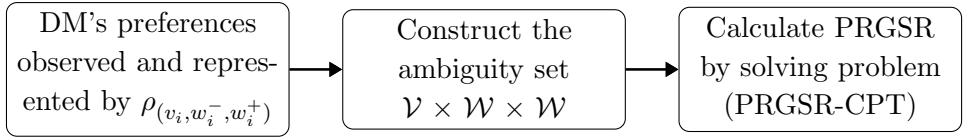

\section{Construction of the ambiguity sets} 
\label{sec3:ambiguity_VW}
A key ingredient of the 
PRGSR-CPT model (\ref{eq:PRGRS-CPT})
is  construction of 
the ambiguity sets of the value functions and weighting functions because it concerns collection of 
a DM's preference information and calculation of 
the robust risk measure. In general, it is 
difficult to construct the ambiguity sets of the value functions and weighting functions simultaneously.
In this section, we first discuss the construction of the ambiguity set ${\cal V}$ 
of the value functions for a fixed weighting function
and then the ambiguity set ${\cal W}$ of weighting functions.
Throughout 
this section,
we make the following assumption.

\begin{assumption} 
\label{Assu:preference}
The decision maker's 
preference can be described by 
GSR-CPT $\rho_{(v,w^-,w^+)}(\xi)$
but there is 
an ambiguity about how to identify the true
value function and weighting function 
such that 
$\rho_{(v,w^-,w^+)}(\xi)$
precisely 
captures the DM's risk preference. 
\end{assumption}

As in the existing works of PRO models (see e.g. \cite{ArD15,HuM15,GuX21}), 
there are two types of information for 
identifying the true unknown value function.
 One is generic information about the shape of the value functions such as convexity, concavity, $S$-shapedness which 
 is shared by a majority of decision makers in most decision making problems. The other is decision maker/problem specific 
information which has to be elicited through various approaches such as pairwise comparisons, certainty equivalents and past decisions.   
 
In contrast, the
ambiguity set of the weighting functions
may be constructed in a different manner 
given the fact that one can usually use 
empirical data or subjective judgement to construct 
a nominal weighting function. 
Due to the limitation of data, the nominal weighting function
might not be accurate and hence one can usually construct a ball of weighting functions centered 
at the nominal as an ambiguity set. The larger the ball, the more likely the true distortion function is included,
 see Wang and Xu \cite{WaX20} for the
 distortion risk measure optimization.

We propose to 
construct the ambiguity sets of value functions and weighting functions
in three steps.
First,
we extend the certainty equivalent approach in terms of utility-based shortfall risk measure to generalized shortfall risk measure based on CPT to translate the preferences relations of GSR-CPT
into inequalities of the value function.
Second,
we use the classical preference elicitation approach, pairwise comparisons, to narrow down the ambiguity set by increasing the number of questionnaires
and follow the experimental result in practise to confine the study to the case that the value function is strictly S-shaped.
Third,
we will consider a nominal weighting function and take into account the $\dd_{\tilde g}$-ball (Definition \ref{defi-Ball-N}) centered at the nominal weighting function with radius.

 \subsection{
 Characterization of the value function
via a DM's certainty equivalent information}
Under Assumption \ref{Assu:preference}, we may 
use $\rho_{(v,w^-,w^+)}$ to describe a 
DM's preferences.
In order to construct an ambiguity set of 
the value functions, denoted by ${\cal V}$,
for a fixed pair of weighting functions $(w^-,w^+)$,
we begin by exploiting the sufficient and necessary conditions of value functions for the elicited information of GSR-CPT, as analyzed in \cite{DGX18},
here we use certainty equivalent approach.
Let 
\bgeqn
{\cal R}:=\{\rho_{(v,w^-,w^+)}: \rho_{(v,w^-,w^+)}(\xi)
\text{ is defined as in (\ref{eq:rho-CPT})}
\}.
\label{defi:mathcal-R-CPT-based-RM}
\edeqn
\begin{definition}[Certainty equivalent]
Let $\{A_k\}_{k=1}^K
\subset {\cal L}^0(\R)
$ be  a list of considered random variables
and $\{a_k^{+}\}_{k=1}^K, \{a_k^{-}\}_{k=1}^K
\subset \R$ be  two lists of non-random variables.
The set ${\cal R}^K_{ce}\subset {\cal R}$ denotes the set of GSR-CPT $\rho_{(v,w^-,w^+)}\in {\cal R}$ satisfying an associated set of confidence intervals $[a_k^{-},
a_k^{+}]\subset [\essinf{A_k},\esssup{A_k}]$ for the ``certainty equivalent" of each $A_k$ for $k=1,\cdots,K$,
i.e.,
\bgeqn
{\cal R}^K_{ce}:=
\left\{\rho_{(v,w^-,w^+)}:{\cal L}^0(\R)\rightarrow \R
\,\big |\, a_k^{-}\leq \rho_{(v,w^-,w^+)}(A_k)\leq a_k^{+},
\; \forall k=1,\cdots,K\right\}.
\label{eq:constraint-calR-pair}
\edeqn
\end{definition}

Note that (\ref{eq:constraint-calR-pair}) describes the case that the risk of each $A_k$ to be lower than $a_k^{+}$ and bigger than $a_k^{-}$.
This motivates us to identify such two bounds for the riskiness of random variable $A_k$ by taking the forms of questions such as \cite{DGX18,WaX21}:
\begin{itemize}
    \item Upper bound $a_k^{+}$: 
    ``What is the smallest amount of money that you would decline to pay instead of being exposed to the risk of $A_k$ ?"
    \item Lower bound $a_k^{-}$: 
    ``What is the largest amount of money that you would be willing to pay instead of being exposed to the risk of $A_k$ ?"
\end{itemize}

Note that 
if the answers to both questions 
give rise to an identical number, i.e., 
$a_k^{-}=a_k^{+}=\bar{a}_k$,
then the certainty equivalent of $A_k$ is identified
with $\rho_{(v,w^-,w^+)}(A_k)=\bar{a}_k$.
In practice, it is more likely to have 
an interval $[a_k^{-},a_k^{+}]$ 
for the equivalent value of random variable $A_k$.
Moreover, we assume that 
$[a_k^{-},a_k^{+}]\subset (\essinf{A_k},\esssup{A_k})$ which 
is not restrictive. This is because
we can substitute $a_k^{-} >\max\{a_k^{-},\essinf{A_k}\}$ and
$a_k^{+}<\min\{a_k^{+}, \esssup{A_k}\}$ 
otherwise.
The above certainty equivalent approach 
is an extension of Delag et al.~\cite{DGX18} under VNM's expected utility theory or 
Wang and Xu \cite{WaX21} under the distortion risk measure 
to CPT based shortfall risk measures. 

${\cal R}^K_{ce}$ defines a set of GSRs each of which can be used to describe the DM's risk preference expressed via certainty equivalent questionnaires. Our interest here is to use this information to identify a class of  value functions ${\cal V}^K_{ce}\subset \mathscr {V} $ such that
${\cal V}^K_{ce}$ can be used for computation of
PRGSR $\rho_{{\cal V}\times{\cal W}}(\xi)$. The next proposition states this.

\begin{proposition}[Characterization of 
${\cal V}$ for $\rho_{{\cal R}^K_{ce}}(\xi)$] 
\label{prop:certainty-equi-rho}
Let
\bgeqn
\rho_{{\cal R}^K_{ce}}(\xi) :=\sup_{\rho_{(v,w^-,w^+)}\in {\cal R}^K_{ce}} \rho_{(v,w^-,w^+)}(\xi). 
\label{def:rho-mathcalR-ce}
\edeqn
Define
\bgeqn
\label{eq:definition-calV-pair}
{\cal V}^K_{ce} :=\left\{v\in \mathscr {V}:\begin{array}{ll}
&\bbe_{w^-w^+}[v(A_k-a_k^{+})]\leq 0,\\
&\bbe_{w^-w^+}[v(A_k-a_k^{-})]\geq 0,
\end{array}
\inmat{ for }  k=1,\cdots,K\right\}.
\edeqn
Then  
\bgeqn
\rho_{(v,w^-,w^+)}\in {\cal R}^K_{ce} 
\quad \Longleftrightarrow\quad
v\in {\cal V}^K_{ce}
\label{eq:equivalent-rho-Vpair-Vpair}
\edeqn
and
\bgeqn
\label{eq:rho-R_pair=Pi_v_pair}
\rho_{ {\cal R}^K_{ce}}(\xi)  
=\rho_{{\cal V}^K_{ce}}(\xi).
\edeqn

\end{proposition}

\begin{proof}  The equality 
in (\ref{eq:rho-R_pair=Pi_v_pair}) follows from 
the equivalence in
(\ref{eq:equivalent-rho-Vpair-Vpair}) 
and Proposition \ref{prop:PGSR-robust-equal}. So it 
is enough to prove (\ref{eq:equivalent-rho-Vpair-Vpair}). 
Let 
\bgeqn
{\cal V}_{{\cal R}^K_{ce}}
:=\{v\in \mathscr {V}: \rho_{(v,w^-,w^+)}\in  {\cal R}^K_{ce}\}.  
\label{eq:defi-calV-R-pair}
\edeqn
We want to show
$
{\cal V}^K_{ce}={\cal V}_{ {\cal R}^K_{ce}}. 
$
\underline{First, we 
show that ${\cal V}^K_{ce}\subseteq {\cal V}_{{\cal R}^K_{ce}}$}. 
Let $\hat v\in {\cal V}^K_{ce}$.
Then $\hat{v}$ satisfies the first inequality in (\ref{eq:definition-calV-pair}), i.e.,
\bgeqn
\bbe_{w^-w^+}[\hat{v}(A_k-a_k^{+})]\leq 0.
\label{eq:bbe-hatv-etak-gammak-geq}
\edeqn
We show that this 
inequality implies the second inequality in (\ref{eq:constraint-calR-pair}),
i.e.,
\bgeqn
\rho_{(\hat v, w^-,w^+)}(A_k)\leq a_k^{+}.
\label{eq:calR-pair-first-inequality}
\edeqn
This is evidenced by the fact that
\bgeqn
\rho_{(\hat v, w^-,w^+)}(A_k)
&=&\inf\{x \in X: \bbe_{w^-,w^+}[\hat{v}(A_k-x)]\leq 0\} \nonumber \\
&=& a_k^{+}+\inf\{x' \in \R
: \bbe_{w^-w^+}[\hat{v}(A_k-a_k^{+}-x')]\leq 0\}\nonumber \\
& \leq
& a_k^{+},
\label{eq:first-constraint-calR-pair}
\edeqn
where the last inequality is due to (\ref{eq:bbe-hatv-etak-gammak-geq}).
Next, since
$\hat v\in {\cal V}^K_{ce}$
also implies
\bgeqn
\bbe_{w^-w^+}[\hat{v}(A_k-a_k^{-})]\geq 0, \;\forall k=1,\cdots,K,
\label{eq:bbe-hatv-etak-gammak-leq}
\edeqn
we will show that inequality (\ref{eq:bbe-hatv-etak-gammak-leq}) implies 
the first inequality in 
(\ref{eq:constraint-calR-pair}), i.e.,
\bgeqn
\rho_{(\hat v, w^-,w^+)}(A_k)\geq a_k^{-}.
\label{eq:calR-pair-second-inequality}
\edeqn
Let 
 $\Omega_k:=\{\omega\in \Omega: A_k(\omega)\leq a_k^{-}\}.
$ 
Since
$a_k^{-}> \essinf A_k$,
then 
$\mathbb{P}(\Omega_k)>0$.
Let $\varepsilon>0$ be a positive number.
By taking $a=a_k^{-}-\varepsilon$
and $b=a_k^{-}$,
 we have from Lemma \ref{lemma:monotone-of-bbewv} (ii) that 
 \bgeq
\bbe_{w^-w^+}[\hat{v}(A_k-a_k^{-}+\varepsilon)]=\bbe_{w^-w^+}[\hat{v}(A_k-a)]
>\bbe_{w^-w^+}[\hat{v}(A_k-b)]
=\bbe_{w^-w^+}[\hat{v}(A_k-a_k^{-})].
 \edeq
Combining this with (\ref{eq:bbe-hatv-etak-gammak-leq}),
we have
\bgeqn
\bbe_{w^-w^+}[\hat v(A_k-a_k^{-}+\varepsilon)]>  0.
\label{eq:strict-etak-gammak-epsilon}
\edeqn
By exploiting (\ref{eq:strict-etak-gammak-epsilon}),
we have that for  $k=1,\cdots,K$,
\bgeqn
\rho_{(\hat v, w^-,w^+)}(A_k)
&=&\inf\{x \in \R: \bbe_{w^-,w^+}[\hat v(A_k-x)]\leq 0\} \nonumber \\
&=& a_k^{-}+\inf\{x' \in \R
: \bbe_{w^-,w^+}[\hat v(A_k-a_k^{-}-x')]\leq 0\}\nonumber \\
&\geq & a_k^{-},
\label{eq:second-constraint-calR-pair}
\edeqn
which proves (\ref{eq:calR-pair-second-inequality}) as desired.
Combining (\ref{eq:calR-pair-first-inequality}), (\ref{eq:calR-pair-second-inequality})
and the definition of ${\cal R}^K_{ce}$ in  (\ref{eq:constraint-calR-pair}),
we 
conclude that
$\rho_{(\hat{v}, w^-,w^+)} \in  {\cal R}^K_{ce}$
and hence $\hat{v}\in {\cal V}_{{\cal R}^K_{ce}}$.
This completes the first part of the proof.

\underline{Next,
we will show that ${\cal V}^K_{ce}\supseteq {\cal V}_{{\cal R}^K_{ce}}$}. 
Let $\hat{v}\in {\cal V}_{{\cal R}^K_{ce}}$,
which means
$\rho_{(\hat v, w^-,w^+)}\in  {\cal R}^K_{ce}$ by (\ref{eq:defi-calV-R-pair}),
that is,
(\ref{eq:calR-pair-second-inequality}) 
and (\ref{eq:calR-pair-first-inequality}) hold.
By exploiting (\ref{eq:calR-pair-first-inequality}),
i.e., $\rho_{(\hat v, w^-,w^+)}(A_k)\leq a_k^{+}$,
we will show (\ref{eq:bbe-hatv-etak-gammak-geq}) hold,
i.e., $\bbe_{w^-w^+}[\hat{v}(A_k-a_k^{+})]\leq 0$.
Since $\rho_{(\hat v, w^-,w^+)}(A_k)$ is the solution of problem (\ref{eq:rho-CPT}) with 
$\xi=A_k$,
we have $\rho_{(\hat v, w^-,w^+)}(A_k)$ satisfies the constraint, 
that is,
\bgeqn
\bbe_{w^-w^+}[\hat v(A_k-\rho_{(\hat v, w^-,w^+)}(A_k))]
\leq 0.
\label{eq:solution-rho-hatv-etak}
\edeqn
Let $\tilde \Omega_k:=\{\omega\in \Omega: A_k(\omega)\leq \rho_{(\hat{v},w^-,w^+)}(A_k)\}$,
$a:=\rho_{(\hat{v},w^-,w^+)}(A_k)\}$
and $b:=a_k^{+}$. Then $a\leq b$, 
Moreover, since $\rho_{(\hat{v},w^-,w^+)}(A_k)\in [a_k^{-},a_k^{+}]\subset 
[\essinf{A_k},\esssup{A_k}]
$
where the inclusion ``$\subset$'' is strict, then
$\mathbb{P}(\tilde \Omega_k)>0$.
It follows from Lemma \ref{lemma:monotone-of-bbewv} (i) that 
 \bgeq 
 \bbe_{w^-w^+}[\hat v(A_k-\rho_{(\hat v, w^-,w^+)}(A_k))]
=\bbe_{w^-w^+}[\hat{v}(A_k-a)]
\geq \bbe_{w^-w^+}[\hat{v}(A_k-b)] 
 =\bbe_{w^-w^+}[\hat v(A_k-a_k^{+})].
 \edeq

Likewise, by exploiting  (\ref{eq:calR-pair-second-inequality})
i.e., $\rho_{(\hat v, w^-,w^+)}(A_k)\geq a_k^{-}$,
we 
prove (\ref{eq:bbe-hatv-etak-gammak-leq}) holds,
i.e., 
$\bbe_{w^-,w^+}[\hat v(A_k-a_k^{-})] 
\geq 0,\; \inmat{for}\; k=1,\cdots,K. 
$
Since $\rho_{(\hat v, w^-,w^+)}(A_k)$ is the solution of problem (\ref{eq:rho-CPT}) with 
$\xi=A_k$,
we have $\rho_{(\hat v, w^-,w^+)}(A_k)$ satisfies the indifference equation,
that is,
\bgeqn
\bbe_{w^-w^+}[\hat v(A_k-\rho_{(\hat v, w^-,w^+)}(A_k))]
= 0.
\label{eq:solution-rho-hatv-etak-}
\edeqn
Moreover, 
set $a=a_k^{-}$ and $b=\rho_{(\hat{v},w^-,w^+)}(A_k)$.
Since 
$a\leq b$,  
it follows by Lemma \ref{lemma:monotone-of-bbewv} (i) that  \bgeq 
 \bbe_{w^-w^+}[\hat v(A_k-\rho_{(\hat v, w^-,w^+)}(A_k))]
=\bbe_{w^-w^+}[\hat{v}(A_k-b)]
\leq \bbe_{w^-w^+}[\hat{v}(A_k-a)]
 =\bbe_{w^-w^+}[\hat v(A_k-a_k^{-})].
 \edeq
 Combining with (\ref{eq:solution-rho-hatv-etak-}), 
we have
(\ref{eq:bbe-hatv-etak-gammak-leq}) hold.

Thus by using (\ref{eq:calR-pair-second-inequality}) 
and (\ref{eq:calR-pair-first-inequality}),
we derive (\ref{eq:bbe-hatv-etak-gammak-geq}) and (\ref{eq:bbe-hatv-etak-gammak-leq}),
which means that 
for every $\hat v \in {\cal V}_{{\cal R}^K_{ce}}$,
we have $\hat{v}\in {\cal V}^K_{ce}$.
Thus ${\cal V}_{{\cal R}^K_{ce}}\subseteq {\cal V}^K_{ce}$.
The proof is completed.
\hfill $\Box$
\end{proof}

Representing $\rho_{ {\cal R}^K_{ce}}(\xi)$
 as $\rho_{{\cal V}^K_{ce}}(\xi)$ is useful for solving (PRGSR-CPT) since it reduces the evaluation of the PRGSR-CPT to finding the optimal value of a stochastic programming problem with semi-infinite stochastic constraints:
 \bgeqn
\rho_{ {\cal R}^K_{ce}}(\xi)
=\inf\limits_{x\in \R} \{ x: \bbe_{w^-w^+}[v(\xi-x)]\leq0, \forall v\in {\cal V}^K_{ce}\}.
 \edeqn
 In Section \ref{sec:Computational-schemes},
 we will address the computational challenges arising from this reformulation for the case that $\xi$ has a finite discrete distribution.

\subsection{Construction of
${\cal V}$ via 
pairwise comparison approach}
\label{subsec:construct-pariwise-comparison}

In the literature of behavioural economics and preference robust optimization,
a widely used approach for eliciting a decision maker's preference is 
pairwise comparison, that is, the decision maker is given a pair of prospects 
for choice and the outcome of the choice is used to infer the decision maker's 
utility/risk preference. Here we use the same approach for constructing an ambiguity set of 
the value functions 
assuming the weighting functions are known and fixed. Differing from the discussions in the 
preceding subsection, here 
we will not use 
CPT-based shortfall risk measure to represent the DM's preference choice,
rather 
we will use the distorted expected values of prospects, that is,  $\bbe_{w^-w^+}[v(\cdot)]$
to describe the DM's choice.
For example, if the DM's is found to prefer prospect $B$ over $A$, then we claim that 
 $\bbe_{w^-w^+}[v(A)]\leq  \bbe_{w^-w^+}[v(B)]$. This is because we are unable to 
 use $\rho_{(v,w^-,w^+)}(A)\leq  \rho_{(v,w^-,w^+)}(B)$ to infer the corresponding properties of the unknown 
 true value function $v(\cdot)$, see \cite{DGX18} for the same issues in 
 the preference robust SR models. To justify this approach, 
 we note that when DM prefers $B$ to $A$, 
he/she would 
also prefer $B-t$ 
to $A-t$ for any constant $t$. 
Consequently
we have 
$$
\bbe_{w^-w^+}[v(A-t)]\leq  \bbe_{w^-w^+}[v(B-t)], \forall t\in \R.
$$
This inequality implies $\rho_{(v,w^-,w^+)}(A)\leq  \rho_{(v,w^-,w^+)}(B)$. 
The converse argument might not hold.
So we may regard the preference represented by 
functional $\bbe_{w^-w^+}[v(\cdot)]$ as CPT-based 
preference representation, in other words, 
the DM's preference in this case can be represented by a ``stronger'' quantitative indicator.

Let $\{\gamma_m, \eta_m\}_{m=1}^{M}$ be a set of random variable. Suppose that the DM  is found to prefer $\eta_m$ over $\gamma_m$ for $m=1,\cdots, M$.
Based on the discussions above, 
we can use the preference information to define 
an
ambiguity set of the value functions 
\bgeqn 
{\cal V}^M_{pair}:=
\left\{v\in {\mathscr V}: \mathbb E_{w^-w^+}[v(\gamma_m)]\leq \mathbb E_{w^-w^+}[v(\eta_m)],\; m=1,\cdots,M \right\}.
\edeqn

In some cases, it might be possible to know roughly the shape of the value function such as convexity, concavity and $S$-shapedness. In the literature of behavioural economics, $S$-shapedness of a value function is widely adopted.  
Let 
\bgeqn 
\mathcal V_S :=\{v\in \mathscr V: v^+ \mbox{ is concave, } v^- \mbox{ is convex}\}.
\edeqn
This kind of information can usually be more easily elicited. For example, if a DM is found to be risk taking in losses and risk averse in gains, then the value function should be $S$-shaped.

Summarizing the elicitation approaches that we have discussed, we may define 
the ambiguity set of the value functions broadly as
\bgeqn
{\cal V}_{El}:={\cal V}^K_{ce}\cap {\cal V}^M_{pair}\cap {\cal V}_{S}
\label{eq:V-el}
\edeqn 
albeit in practice we may only use part of them.

\subsection{Construction of ambiguity set ${\cal W}$ of weighting functions}

Next, we discussion how to construct ${\cal W}$.
The weighting function reflects 
the DM's attitudes toward real probabilities of random losses/gains. 
A lot of discussions in behavioural economics are devoted to the topic.
Empirical evidences suggest that
the weighting functions in decision making problems are typically
inverse-S shaped (initially concave 
over interval $(0,p^*)$ for some $p*\in (0,1)$
and then convex over $(p^*,1)$).
The inverse-S shape of weighting function reflects 
that the DM tends to overweight probabilities of extreme outcomes
(including suffering large losses and small losses).
For instance, in insurance pricing problem,
 this is consistent with individuals' demand for insuring large losses
 and the eagerness to insure small losses
 and insurer should pay more attention on these two extreme outcomes.

In this section, we assume that the true 
weighting function of a decision maker is roughly known 
either based on empirical data or based on a subjective judgement. For example, Escobar and Pflug \cite{EsP18}  propose an effective approach to construct a step-like distortion function with empirical data and use it to approximate the true unknown distortion/weighting function in an insurance premium calculation model.
We call such an estimated weighting function as nominal weighting function. 
Instead of using the nominal weighting function 
directly for the calculation of the preference robust shortfall risk in (\ref{eq:PRGRS-CPT}), we take a precaution to construct an ambiguity 
set of weighting functions near the nominal for (\ref{eq:PRGRS-CPT}).
Here we consider a ``ball'' of weighting functions centered at the nominal weighting function under some pseudo-metric.  

\begin{definition}
[Pseudo-metric $\dd_{\tilde g}$]
\label{defi-tilde-g}
Let $$
\mathscr G:=\left\{g:[0,1]\rightarrow \R, g \mbox{ is measurable },
|g(t)|\leq \tilde g(t) \mbox{ for $t\in [0,1]$ with } \int_{0}^{1}\tilde g(t)dt<+\infty\right\}.
$$
The pseudo-metric $\dd_{\mathscr G}$ of weighting functions in the space of $\mathscr W$  is defined as
$$
\dd_{\tilde g}(w_1,w_2):=\sup_{g\in \mathscr G} \left|\int_{0}^{1}g(t)\psi_1(t) dt-\int_{0}^{1}g(t)\psi_2(t) dt\right|,
$$
where $\psi_1$ and $\psi_{2}$ are
derivatives of the weighting functions $w_1$ and $w_2$ respectively.
\end{definition}
Note that in this definition $\dd_{\tilde g}(w_1,w_2)=0$
if and only if $\int_{0}^{1}g(t)\psi_1(t) dt=\int_{0}^{1}g(t)\psi_2(t) dt$ for all $g\in \mathscr{G}$, but $\psi_1$ is not necessarily equal to
$\psi_2$ unless the set $\mathscr{G}$ is sufficiently large. If we interpret $\mathscr{G}$ as a set of test functions, then $\dd_{\tilde g}(w_1,w_2)=0$ means 
that there is no difference between the two weighting functions for the specified set of test functions/cases.
Moreover, we implicitly assume that
the weighting functions are differentiable with at most
a finite number of nondifferentiable points over $(0,1)$.
The reason why we use the pseudo-metric as opposed to other metrics/distance  is because it is easy to use 
particularly in the derivation of tractable formulations of PRGSR-CPT in the next section.

Let $\mathscr W_T\subset \mathscr{W}_{in S}$ be 
the set of 
all piecewise linear 
inverse S-shaped 
functions $w:[0,1]\rightarrow [0,1]$
 with breakpoints 
 $0=t_1<\cdots<t_{T+1}=1$ with 
 $t_{l_{p_*}}=p_*$,
 where 
 $
 l_{p_*}\in\{1,\cdots,T+1\}.
 $
Let 
\bgeqn
\label{defi:Psi}
\Psi_T:=\{\psi:[0,1]\rightarrow \mathbb R_+: 
\psi \inmat{ is decreasing over } (0,p^*), \inmat{ increasing over } (p^*,1)\}
\edeqn
be the set of the derivative functions of $\mathscr W_T$.
Obviously 
$\Psi_T$ is a set of left continuous step-like functions defined over $[0,1]$
with jumps at the breakpoints $t_2,\cdots, t_T$.
Then we can write $\Psi_T$ as:
\bgeqn
\label{defi:Psi-T}
\Psi_T=\left \{\psi:[0,1]\rightarrow \mathbb R_+: 
 \begin{array}{ll}
 &\psi(t)=
 \sum_{l=1}^T \psi_l \mathds{1}_{[t_l,t_{l+1})}(t),\; \sum_{l=1}^{T} \psi_l(t_{l+1}-t_l)=1 \\
&\psi_{l+1}\leq \psi_l, l=1,\cdots,l_{p_*}-2,
\psi_{l}\leq \psi_{l+1}, l=l_{p_*},\cdots,T
  \end{array}
  \right\},\;\;
 \edeqn
 where $\psi_l, l=1,\cdots,T$ are the slopes of linear pieces of $w\in \mathscr W_T$ and
 $ \mathds{1}_{S}$ is the indicator function
 of  set $S$ with $\mathds{1}_{S}(t)=1$ for $t\in S$ and $0$ otherwise.
 \begin{definition}
[$\tilde g$-ball of weighting functions]
\label{defi-Ball-N}
 Let  $w_0\in  {\mathscr W}_T$ be a weighting function which is 
 piecewise linear 
 and inverse S-shaped. 
 The $\tilde g$-ball of inverse $S$-shaped weighting functions 
 centered at $w_0$ with radius $r$  is defined as
\bgeq
B(w_0,r):=\left\{w \in {\mathscr W}_{inS}:
\dd_{\tilde g}(w,w_0)\leq r
\right\},
\edeq
where the radius is related to  
the DM's confidence 
about $w_0$.
\end{definition}

Note that $B(w_0,r)$ may contain inverse $S$-shaped weighting functions which are not piecewise linear. However, we will show later on that the worst case weighting function is piecewise linear, see Theorem \ref{thm:reformulation-constraint-function-VW}. 
Note also that when radius $r$ is driven to $0$, 
$B(w_0,r)$ shrinks to a singleton
$\{w_0\}$ in the case that $\dd_{\tilde g}$ is a full metric.
Based on the above comment, 
in this paper,
for facilitating computation,
we consider the ambiguity set of weighting functions as follows
\bgeqn
\label{eq:Wr}
{\cal W}_{r}=B(w_0,r).
\edeqn
Note that in the case when $\tilde g$ takes a specific function, 
we can derive the $\dd_{\tilde g}$-ball explicitly.

\begin{example}
Let $\tilde g(t)=1$. Then $\tilde g$-ball in Definition \ref{defi-Ball-N} collapses to the following ball in the $L_1$-metric sense:
\bgeqn
\label{ambiguit_W2}
B_{L_1}(w_0,r):=\left\{w \in \mathscr W_{inS}:
 \dd_1(w,w_0)\leq r
\right\},
\edeqn
where $\dd_1(w,w_0):=\int_{0}^{1}|\psi(t)-\psi^0(t)|dt$.
\end{example}

\section{Computational schemes}
\label{sec:Computational-schemes}
In this section,
we discuss how to calculate 
PRGSR-CPT 
$\rho_{{\cal V}\times {\cal W}}(\xi)$ for a given prospect $\xi$ when the ambiguity sets ${\cal V}$ and ${\cal W}$ are given.
Specifically, we consider how to solve the following minimization 
problem efficiently
\bgeqn
 \label{eq:robust-V-W-computation}
\inmat{(PRGSR-CPT-S)}\quad &\inf\limits_{x\in \R}& x \nonumber \\
 &{\rm s.t.}&  \sup_{(v,w^-,w^+)\in {\cal V}_{El}\times {\cal W}_r\times {\cal W}_r} \bbe_{w^-w^+}[v(\xi-x)] \leq 0,
\edeqn
where ${\cal V}_{El}$ and ${\cal W}_r$ are 
 defined in (\ref{eq:V-el}) and (\ref{eq:Wr}) respectively. We use (PRGSR-CPT-S) to indicate 
 that this is a (PRGSR-CPT) when the ambiguity sets
 take a specific structure. Throughout this section, we confine our discussions to the case that the random variable $\xi$ is finitely discretely distributed.
In the case that the distribution is continuous,
we can use the sample average approximation method to discretize, 
see, e.g. Guo and Xu \cite{GuX21}.
Problem (\ref{eq:robust-V-W-computation})
is mathematical program with
semi-infinite constraints.
To tackle the problem, we 
propose to derive a dual formulation of the constraint function.
To this end, we need to introduce the following notation.

Let $\Xi :=\{\xi_1,\cdots,\xi_N\}\subset \R$ denote the support set of $\xi$ with $P(\xi=\xi_i)=p_i$ for $i=1,\cdots,N$.
By sorting out the elements in $\Xi$  in 
an increasing order 
($\xi_{j_1}\leq \xi_{j_2}$, $j_1<j_2$,),
we obtain 
$
\xi_{j_1}-x\leq \xi_{j_2}-x \mbox{ for }j_1<j_2.
$
By convention,
for 
the outcomes $\xi_{j}-x<0$, we label them with indexes 
$j\in \{-m,-m+1,\cdots,-1\}$
and 
for those $\xi_j-x\geq 0$, we label them
by $j\in \{0,\cdots,n\}$.
Consequently we can rewrite the integral 
$
\bbe_{w^-w^+}[v(\xi-x)]=\int_{-\infty}^{0} v^-(t) dw^-(F_{\xi-x}(t))+\int_{0}^{\infty} v^+(t) d(1-w^+(1-F_{\xi-x}(t)))
$
as 
\bgeqn
\label{saa_ww}
&&\bbe_{w^-w^+}[v(\xi-x)]= \sum_{i=-m}^n \pi_i v(\xi_i-x),
\edeqn
where $v(\xi_i-x) =v^-(\xi_i-x)$ 
for $i=-m,\cdots,-1$,
and $v(\xi_i-x) =v^+(\xi_i-x)$ for $i=0,\cdots,n$,
and
\bgeqn
\label{pi_omega}
\pi_i :=\left\{
\begin{array}{ll}
\displaystyle 
w^-(p_{-m})& \inmat{for} \; i=-m\\[4pt]
\displaystyle w^-\left(\sum_{j=-m}^ip_j\right)-w^-\left(\sum_{j=-m}^{i-1}p_j\right)    & \inmat{for} \; i=-m+1,\cdots,-1\\[4pt]
\displaystyle 
w^+\left(\sum_{j=i}^{n}p_j\right)-w^+\left(\sum_{j=i+1}^np_j\right)    & \inmat{for} \; i=0,1,\cdots, n-1 \\[4pt]
w^+(p_{n})& \inmat{for} \; i=n.\\[4pt]
\end{array}
\right.
\edeqn
In the case that $w^+(p) = 1-w^-(1-p), \forall p\in[0,1]$,
we have $\sum_{i=-m}^n \pi_i=1$,
and for $i=0,\cdots,n-1,$
\bgeq
\pi_i=1-w^-\left(1-\sum_{j=i}^{n}p_j\right)-\left(1-w^-\left(1-\sum_{j=i+1}^np_j\right)\right)
=w^-\left(\sum_{j=-m}^ip_j\right)-w^-\left(\sum_{j=-m}^{i-1}p_j\right).   
\edeq
Consequently, $\pi_i$ can be expressed
$\pi_i=w^-\left(\sum_{j=-m}^ip_j\right)-w^-\left(\sum_{j=-m}^{i-1}p_j\right), \text{for} \; i=-m+1,\cdots, n$.

\subsection{
Reformulation of (PRGSR-CPT-S)}
With the notation introduced above, we can recast 
(PRGSR-CPT-S) as (PRGSR-CPT-S-Dis) (where ``Dis'' is used to indicate the case that $\xi$ is discretely distributed):
\bgeqn
 \label{eq:robust-V-W-computation-SAA}
\inmat{(PRGSR-CPT-S-Dis)}\quad 
\rho_{{\cal V}_{El} \times {\cal W}_r}(\xi)
=&\inf\limits_{x\in \R}& x \nonumber \\
 &{\rm s.t.}&  \sup_{(v,w^-,w^+)\in {\cal V}_{El} \times {\cal W}_r\times {\cal W}_r} \sum_{i=-m}^n \pi_i v(x-\xi_i) \leq 0.
\edeqn
In this section,
we confine our study to the case that value function is defined over $[\alpha,\beta]$ and take fixed values at points $\alpha,\beta$ with $\alpha<0<\beta$.
The next theorem states that 
the constraint in program (PRGSR-CPT-S-Dis)
can be computed by solving a finite dimensional 
biconvex program of reasonable size.
\begin{theorem}
\label{thm:reformulation-constraint-function-VW}
For 
any fixed $x$,
the constraint function 
$$
\sup_{(v,w^-,w^+)\in {\cal V}_{El}\times {\cal W}_r \times {\cal W}_r} \sum_{i=-m}^n \pi_i v(\xi_i-x)
$$ 
in (\ref{eq:robust-V-W-computation-SAA})
is the optimal value of the following biconvex program:
\begin{subequations}
\label{robust_v_o}
\begin{align}
 & \sup_{a,b,\psi^-, \psi^+, \eta^-,\theta^-, \eta^+,\theta^+} \left\{\sum_{i=-m}^{-1} 
\Big( (a_1(\xi_i-x)+b_1)\mathds{1}_{(-\infty,0)\cap[y_1,y_2]}(\xi_i-x) \right. \nonumber \\
& \qquad \qquad \qquad  +\sum_{j=2}^{J-1}(a_j(\xi_i-x)+b_j)\mathds{1}_{(-\infty,0)\cap(y_j,y_{j+1}]}(\xi_i-x) \Big) \sum_{l=1}^T\psi_l^-\int_{q_{i-1}}^{q_{i}}\mathds{1}_{[t_l,t_{l+1})}(t)dt \nonumber \\
& \qquad \qquad \qquad 
+\sum_{i=0}^{n} \Big(
(a_1(\xi_i-x)+b_1)\mathds{1}_{(-\infty,0)\cap[y_1,y_2]}(\xi_i-x)\nonumber  \\
& \left. \qquad \qquad \qquad 
+\sum_{j=2}^{J-1}(a_j(\xi_i-x)+b_j)\mathds{1}_{(-\infty,0)\cap(y_j,y_{j+1}]}(\xi_i-x) \Big)
\sum_{l=1}^T\psi^+_l\int_{1-q_{i}}^{1-q_{i-1}}\mathds{1}_{[t_l,t_{l+1})}(t)dt
\right\}\nonumber \\
&\qquad  {\rm s.t.}\qquad \;\;
\psi^-_{l+1}\leq \psi^-_{l}, l=1,\cdots,l_{p_*}-2, \label{robust_Wra} \\
 &\qquad \qquad \qquad
 \psi^+_{l+1}\leq \psi^+_{l}, l=1,\cdots,l_{p_*}-2, \label{robust_Wrb} \\
&\qquad \qquad \qquad
\psi^-_l\leq \psi^-_{l+1},  l=l_{p_*},\cdots,T-1,\label{robust_Wrc} \\
 &\qquad \qquad \qquad
 \psi^+_l\leq \psi^+_{l+1}, l=l_{p_*},\cdots,T-1,\label{robust_Wrd} \\
&\qquad \qquad \qquad
\sum_{l=1}^{T} \psi^-_l(t_{l+1}-t_l)=1,\label{robust_Wre} \\
&\qquad \qquad \qquad
 \sum_{l=1}^{T} \psi^+_l(t_{l+1}-t_l)=1,\label{robust_Wrf} \\
& \qquad \qquad \qquad \sum_{l=1}^T(\eta_l^-+\theta_l^-) \int_{t_l}^{t_{l+1}} \tilde g(t) dt\leq r,\label{robust_Wrg} \\
& \qquad \qquad \qquad \sum_{l=1}^T(\eta_l^++\theta_l^+) \int_{t_l}^{t_{l+1}} \tilde g(t) dt\leq r,\label{robust_Wrg-2} \\
& \qquad \qquad \qquad \psi^-_l-\psi^0_{l}+\eta_l^--\theta_l^- =0, l=1,\cdots,T, \label{robust_Wrh} \\
& \qquad \qquad \qquad \psi^+_l-\psi^0_{l}+\eta_l^+-\theta_l^+ =0, l=1,\cdots,T, \label{robust_Wri} \\
& \qquad \qquad \qquad\psi_l^- > 0,\psi_l^+ > 0, \eta_l^-,\theta_l^-,\eta_l^+,\theta_l^+\geq 0, l=1,\cdots, T,\label{robust_Wrj} \\
& \qquad \qquad \qquad \sum_{j=1}^{J-1} a_j \int_{y_j}^{y_{j+1}} \phi_m(t)dt \leq 0, m=1,\cdots, M,\label{robust_v_ob}\\
& \qquad \qquad \qquad 
a_{j-1}y_j +b_{j-1}=a_jy_j+b_j, j=2,\cdots, J-1,\label{robust_v_oc}\\
& \qquad \qquad \qquad 
a_j\leq a_{j+1}, j=1, \cdots, j_{0}-2,\label{robust_v_od}\\
& \qquad \qquad \qquad 
a_{j+1}\leq a_j, j=j_{0},\cdots, J-2, \label{robust_v_oe}\\
& \qquad \qquad \qquad 
a_{j_{0}}0+b_{j_{0}}=0,\;
a_1y_1+b_1=b_l,\;
a_{J-1}y_{J}+b_{J-1}=b_r,\label{robust_v_oh}\\
& \qquad \qquad \qquad  
 \hat{c}^k\leq 0,\;\; 
\tilde c^k\geq 0, k=1,\cdots,K, \\
&\qquad \qquad \qquad  a_j>0, j=1,\cdots,J, \label{robust_v_oj} 
\end{align}
\end{subequations}
where $j_{0}:=\{j\in \{1,\cdots,J\}: y_{j_{0}}=0\}$,
\bgeq
 \hat c^k&=& 
\Big( (a_1(A_k^1-a_k^{+})+ b_1)\mathds{1}_{(-\infty,0)\cap[y_1,y_2]}(A_k^1-a_k^{+}) \\
&&  +\sum_{j=2}^{J-1}( a_j(A_k^1-a_k^{+})+ b_j)\mathds{1}_{(-\infty,0)\cap(y_j,y_{j+1}]}(A_k^1-a_k^{+}) \Big)
\left(w_*^-\left(q^1_k\right)-w_*^-\left(q^{0}_k\right)\right) \\
&&
+\Big(
(a_1(A_k^2-a_k^{+})+ b_1)\mathds{1}_{(0,\infty)\cap[y_1,y_2]}(A_k^2-a_k^{+}) \\
&& 
+ \sum_{j=2}^{J-1}( a_j(A_k^2-a_k^{+})+ b_j)\mathds{1}_{(0,\infty)\cap(y_j,y_{j+1}]}(A_k^2-a_k^{+})\Big)
\left(w_*^+\left(1-q^1_{k}\right)-w_*^+\left(1-q^{2}_k\right) \right),
\edeq
\bgeq
\tilde c^k &=&
\Big( ( a_1(A_k^1-a_k^{-})+ b_1)\mathds{1}_{(-\infty,0)\cap[y_1,y_2]}(A_k^1-a_k^{-}) \\
&&  +\sum_{j=2}^{J-1}( a_j(A_k^1-a_k^{-})+ b_j)\mathds{1}_{(-\infty,0)\cap(y_j,y_{j+1}]}(A_k^1-a_k^{-}) \Big)
\left(w_*^-\left(q^1_k\right)-w_*^-\left(q^{0}_k\right)\right)\\
&&
+\Big(
(a_1(A_k^2-a_k^{-})+ b_1)\mathds{1}_{(0,\infty)\cap[y_1,y_2]}(A_k^2-a_k^{-}) \\
&& 
+\sum_{j=2}^{J-1}( a_j(A_k^2-a_k^{-})+ b_j)\mathds{1}_{(0,\infty)\cap(y_j,y_{j+1}]}(A_k^2-a_k^{-}) \Big)\left(w_*^+\left(1-q^{1}_k\right)-w_*^+\left(1-q^{2}_k\right)\right),
\edeq
$q_{-m-1}:=0,$ $q_i:=\sum_{j=-m}^{i}p_j$ for $i=-m,\cdots,n$,
$q_k^{0}:=0$,
 $q_k^{1}:=1-p_k$,
$q_k^{2}:=1$, 
$A_k^1=0,A_k^2=r_k$ ($p_k$ and $r_k$ are defined as in Section \ref{section-construct-V}, Q2)
and $\{y_j\}_{j=1}^{J}$ is the $j$-th entry of ranked support set 
\bgeqn 
\mathscr S &:=& \bigcup_{m=1}^M( \supp(\gamma_m)\cup \supp(\eta_m))\bigcup\{\alpha,\beta\}\bigcup \{0\}\nonumber\\
&&\quad \bigcup \cup_{i=1}^N(\xi_i-x)\bigcup \cup_{k=1}^K\supp(A_k-a_k^{+})\bigcup \cup_{k=1}^K\supp(A_k-a_k^{-}),
\label{eq:S-support}
\edeqn
and $J:=|\mathscr S|=3M+3+N+4K$.
$\phi_m$ in condition (\ref{robust_v_ob}) is defined as in Section \ref{section-construct-V}, Q1.
Condition (\ref{robust_v_ob}) characterizes the DM's preference between variables $\gamma_m$ and $\eta_m$,
$m=1,\cdots,M$,
conditions (\ref{robust_v_oc}), (\ref{robust_v_od}) and (\ref{robust_v_oe}) represents the increasing of value function $v$ 
and convexity of $v^-$ and concavity of $v^+$,
conditions 
(\ref{robust_v_oh}) represents three fixed values of $v$.
\end{theorem}

The proof is a bit long, we defer it to 
Appendix \ref{app:proof-theorem4-1}.

In general, it is difficult to solve the biconvex program (\ref{robust_v_o}).
Haskell et al.~\cite{HFD16} consider a preference robust optimization model  
where the ambiguity lies not only in 
the DM' risk preferences 
but also in the 
probability distribution of exogenous uncertainties.
The resulting minmax optimization model is a bilinear programming problem.
The authors show how to solve their problem directly 
in a special case when
the ambiguity set of the probability 
distributions can be represented by the 
mixtures of some given distributions. 
However, the bilinear programming problem is generally difficult to solve.
They 
propose 
two 
techniques 
to address 
the difficulty:
a reformulation-linearization technique (RLT) and an semidefinite programming (SDP) relaxation approach.
Both are approximation schemes and 
may introduce many variables
if we 
use them 
to deal with the biconvex structure 
in (\ref{robust_v_o}) directly.
This is undesirable particularly when $N$ is large (
$\xi$ has many scenarios).

Here, we take a different approach. 
We can reformulate program (\ref{robust_v_o})
as a new linear program 
by changing some variables as follows:
\bgeq
\tilde a_j^{l-} :=a_j \psi_l^-,
\tilde a_j^{l+} :=a_j \psi_l^+,
\tilde b_j^{l-} :=b_j \psi_l^-,
\tilde b_j^{l+} :=b_j \psi_l^+,
\; \inmat{for} \; j=1,\cdots,J,
l=1,\cdots,T.
\edeq
that is,
\begin{subequations}
\label{robust_v_o-LP}
\begin{align}
 & \sup_{\tilde a^-,\tilde b^-,\tilde a^+,\tilde b^+,\psi^-,\psi^+, \eta^-,\theta^-,\eta^+,\theta^+} \left\{\sum_{i=-m}^{-1} \sum_{l=1}^T
\Big( (\tilde a_1^{l-}(\xi_i-x)+\tilde b_1^{l-})\mathds{1}_{(-\infty,0)\cap[y_1,y_2]}(\xi_i-x) \right. \nonumber \\
& \qquad \qquad \qquad  +\sum_{j=2}^{J-1}(\tilde a_j^{l-}(\xi_i-x)+\tilde b_j^{l-})\mathds{1}_{(-\infty,0)\cap(y_j,y_{j+1}]}(\xi_i-x) \Big)
\int_{q_{i-1}}^{q_{i}}\mathds{1}_{[t_l,t_{l+1})}(t)dt \nonumber \\
& \qquad \qquad \qquad 
+\sum_{i=0}^{n}\sum_{l=1}^T \Big(
(\tilde a_1^{l+}(\xi_i-x)+\tilde b_1^{l+})\mathds{1}_{(0,\infty)\cap[y_1,y_2]}(\xi_i-x)\nonumber  \\
& \left. \qquad \qquad \qquad 
+\sum_{j=2}^{J-1}(\tilde a_j^{l+}(\xi_i-x)+\tilde b_j^{l+})\mathds{1}_{(0,\infty)\cap(y_j,y_{j+1}]}(\xi_i-x) \Big)
\int_{1-q_{i}}^{1-q_{i-1}}\mathds{1}_{[t_l,t_{l+1})}(t)dt\right\}\nonumber \\
&\qquad  {\rm s.t.}\qquad (\ref{robust_Wra})-(\ref{robust_Wrj}), \nonumber \\ 
& \qquad \qquad \qquad \sum_{j=1}^{J-1} \tilde a_j^{l-} \int_{y_j}^{y_{j+1}} \phi_m(t)dt \leq 0, m=1,\cdots, M,l=1,\cdots,T,\label{robust_v_ob-LP}\\
& \qquad \qquad \qquad \sum_{j=1}^{J-1} \tilde a_j^{l+} \int_{y_j}^{y_{j+1}} \phi_m(t)dt \leq 0, m=1,\cdots, M,l=1,\cdots,T,\label{robust_v_ob-LP-2}\\
& \qquad \qquad \qquad 
\tilde a_{j-1}^{l-}y_j+\tilde b_{j-1}^{l-}=\tilde a_{j}^{l-}y_j+\tilde b_{j}^{l-}, j=2,\cdots, J-1, l=1,\cdots,T,\label{robust_v_oc-LP}\\
& \qquad \qquad \qquad 
\tilde a_{j-1}^{l+}y_j+\tilde b_{j-1}^{l+}=\tilde a_{j}^{l+}y_j+\tilde b_{j}^{l+}, j=2,\cdots, J-1, l=1,\cdots,T,\label{robust_v_oc-LP-2}\\
& \qquad \qquad \qquad 
\tilde a_{j}^{l-}\leq \tilde a_{j+1}^{l-}, 
j=1, \cdots, j_{0}-1,
l=1,\cdots,T,\label{robust_v_od-LP}\\
& \qquad \qquad \qquad 
\tilde a_{j}^{l+}\leq \tilde a_{j+1}^{l+},
j=1, \cdots, j_{0}-1,
l=1,\cdots,T,\label{robust_v_od-LP-2}\\
& \qquad \qquad \qquad 
\tilde a_{j+1}^{l-}\leq \tilde a_{j}^{l-},
j=j_{0},\cdots, J-2,
l=1,\cdots,T,
\label{robust_v_oe-LP}\\
& \qquad \qquad \qquad 
\tilde a_{j+1}^{l+}\leq \tilde a_{j}^{l+},
j=j_{0},\cdots, J-2,
l=1,\cdots,T,
\label{robust_v_oe-LP-2}\\
& \qquad \qquad \qquad
\tilde a_{j_{0}}^{l-}0+\tilde b_{j_{0}}^{l-}=0,
l=1,\cdots,T, \label{robust_v_of-LP}\\
& \qquad \qquad \qquad
\tilde a_{j_{0}}^{+}0+\tilde b_{j_{0}}^{l+}=0,
l=1,\cdots,T, \label{robust_v_of-LP-2}\\
&\qquad \qquad \qquad   \tilde a_{1}^{l-}y_1+\tilde b_{1}^{l-}=\psi_l^-b_l,
l=1,\cdots,T,\label{robust_v_og-LP} \\
&\qquad \qquad \qquad  
\tilde a_{1}^{l+}y_1+\tilde b_{1}^{l+}=\psi_l^+b_l,
l=1,\cdots,T,\label{robust_v_og-LP-2} \\
&\qquad \qquad \qquad 
a_{J-1}^{l-}y_{J}+b_{J-1}^{l-}=\psi_l^-b_r,
l=1,\cdots,T,\label{robust_v_oh-LP}\\
&\qquad \qquad \qquad 
a_{J-1}^{l+}y_{J}+b_{J-1}^{l+}=\psi_l^+b_r,
l=1,\cdots,T,\label{robust_v_oh-LP-2}\\
&\qquad \qquad \qquad  \hat c^{l-}_k\geq 0,\;\;
\tilde c^{l-}_k\leq0,k=1,\cdots,K,\label{robust_v_oj-LP}\\
&\qquad \qquad \qquad  \hat c^{l+}_k\geq 0,\;\;
\tilde c^{l+}_k\leq0,k=1,\cdots,K,\label{robust_v_oj-LP-2}\\
&\qquad \qquad \qquad  a_j^{l-}>0, a_j^{l+}>0, j=1,\cdots,J, l=1,\cdots,T,\label{robust_v_ok-LP}
\end{align}
\end{subequations}
where 
$q_{-m-1}=0,$ $q_i=\sum_{j=-m}^{i}p_j$ for $i=-m,\cdots,n$,
$\tilde a^-:=(\tilde a_{j}^{l-})_{J\times T}$,
$\tilde b^-:=(\tilde b_{j}^{l-})_{J\times T}$,
$\tilde a^+:=(\tilde a_{j}^{l+})_{J\times T}$,
$\tilde b^+:=(\tilde b_{j}^{l+})_{J\times T}$,
 $\tilde \psi^-:=(\psi_1^-,\cdots,\psi_T^-)$, 
  $\tilde \psi^+:=(\psi_1^+,\cdots,\psi_T^+)$, 
 $\tilde \psi^0:=(\psi^0_{1},\cdots,\psi^0_{T})$,
 $q_k^{0}=0$,
 $q_k^{1}=1-p_k$,
$q_k^{2}=1$, 
$A_k^1=0,A_k^2=r_k$,
\bgeq
 \hat c^{l-}_k&=& 
\Big( (\tilde a_1^{l-}(A_k^1-a_k^{+})+ \tilde b_1^{l-})\mathds{1}_{(-\infty,0)\cap[y_1,y_2]}(A_k^1-a_k^{+}) \\
&&  +\sum_{j=2}^{J-1}( \tilde a_j^{l-}(A_k^1-a_k^{+})+ \tilde b_j^{l-})\mathds{1}_{(-\infty,0)\cap(y_j,y_{j+1}]}(A_k^1-a_k^{+}) \Big)
\left(w_*^-\left(q^1_k\right)-w_*^-\left(q^{0}_k\right)\right) \\
&&
+\Big(
(\tilde a_1^{l-}(A_k^2-a_k^{+})+ \tilde b_1^{l-})\mathds{1}_{(0,\infty)\cap[y_1,y_2]}(A_k^2-a_k^{+}) \\
&& 
+ \sum_{j=2}^{J-1}( \tilde a_j^{l-}(A_k^2-a_k^{+})+ \tilde b_j^{l-})\mathds{1}_{(0,\infty)\cap(y_j,y_{j+1}]}(A_k^2-a_k^{+})\Big)
\left(w_*^+\left(1-q^1_{k}\right)-w_*^+\left(1-q^{2}_k\right) \right),
\edeq
\bgeq
\tilde c^{l-}_k &=&
\Big( ( \tilde a_1^{l-}(A_k^1-a_k^{-})+ \tilde b_1^{l-})\mathds{1}_{(-\infty,0)\cap[y_1,y_2]}(A_k^1-a_k^{-}) \\
&&  +\sum_{j=2}^{J-1}( \tilde a_j^{l-}(A_k^1-a_k^{-})+ \tilde b_j^{l-})\mathds{1}_{(-\infty,0)\cap(y_j,y_{j+1}]}(A_k^1-a_k^{-}) \Big)
\left(w_*^-\left(q^1_k\right)-w_*^-\left(q^{0}_k\right)\right)\\
&&
+\Big(
(\tilde a_1^{l-}(A_k^2-a_k^{-})+ \tilde b_1^{l-})\mathds{1}_{(0,\infty)\cap[y_1,y_2]}(A_k^2-a_k^{-}) \\
&& 
+\sum_{j=2}^{J-1}(\tilde a_j^{l-}(A_k^2-a_k^{-})+ \tilde b_j^{l-})\mathds{1}_{(0,\infty)\cap(y_j,y_{j+1}]}(A_k^2-a_k^{-}) \Big)\left(w_*^+\left(1-q^{1}_k\right)-w_*^+\left(1-q^{2}_k\right)\right),
\edeq
and $\hat c^{l+}_k$ in (\ref{robust_v_oj-LP-2}) denotes the counterpart of $\hat c^{l-}_k$ with $\tilde a_j^{l-}$ and $\tilde b_j^{l-}$ being replaced by $\tilde a_j^{l+}$ and $\tilde b_j^{l+}$.
$\tilde c^{l+}_k$ in (\ref{robust_v_oj-LP-2}) denotes the counterpart of $\tilde c^{l-}_k$ with $\tilde a_j^{l-}$ and $\tilde b_j^{l-}$ being replaced by $\tilde a_j^{l+}$ and $\tilde b_j^{l+}$.

Moreover,
for the optimal solution of program (\ref{robust_v_o}),
we can solve (\ref{robust_v_o-LP}) to obtain the optimal solution 
\bgeq
\tilde a_j^{*l-},
\tilde a_j^{*l+},
\tilde b_j^{*l-},
\tilde b_j^{*l+},
\psi_l^{*-},
\psi_l^{*+},
 \eta^{*-},\theta^{*-},
  \eta^{*+},\theta^{*+},
j=1,\cdots,J,
l=1,\cdots,T,
\edeq
then we can get the optimal solution of (\ref{robust_v_o})
\bgeq
a_j^*=\frac{\tilde a_j^{*l-}}{\psi_l^{*-}},
b_j^*=\frac{\tilde b_j^{*l-}}{\psi_l^{*-}},
\psi_l^{*-},
\psi_l^{*+},
 \eta^{*-},\theta^{*-},
  \eta^{*+},\theta^{*+},
j=1,\cdots,J,
\inmat{ for some } l\in\{1,\cdots,T\}.
\edeq
This method will be used in Section \ref{sec5:case_study} for numerical tests.

From Proposition \ref{prop:PGSR-robust-equal},
we claim that the proposed PRGSR-CPT (\ref{eq:robust-V-W-computation-SAA})
can be solved by using the following two-stage procedures:

\begin{algorithm}[Two-stage procedure for solving the proposed PRGSR-CPT (\ref{eq:robust-V-W-computation-SAA})]
\label{algorithm:thm4-1}
\
\begin{itemize}
\item[Step 1.]
For each fixed $x$,
solve linear program:
\bgeqn
\label{eq:two-stage-reformulate-tildephi}
 h(x):=
 && \sup_{\tilde a^-,\tilde b^-,\psi^-,\psi^+, \eta,\theta} \left\{\sum_{i=-m}^{-1} \sum_{l=1}^T
\Big( (\tilde a_1^{l-}(\xi_i-x)+\tilde b_1^{l-})\mathds{1}_{(-\infty,0)\cap[y_1,y_2]}(\xi_i-x) \right. \nonumber \\
&& \qquad \qquad \qquad  +\sum_{j=2}^{J-1}(\tilde a_j^{l-}(\xi_i-x)+\tilde b_j^{l-})\mathds{1}_{(-\infty,0)\cap(y_j,y_{j+1}]}(\xi_i-x) \Big)
\int_{q_{i-1}}^{q_{i}}\mathds{1}_{[t_l,t_{l+1})}(t)dt \nonumber \\
&& \qquad \qquad \qquad 
+\sum_{i=0}^{n}\sum_{l=1}^T \Big(
(\tilde a_1^{l+}(\xi_i-x)+\tilde b_1^{l+})\mathds{1}_{(0,\infty)\cap[y_1,y_2]}(\xi_i-x)\nonumber  \\
&& \left. \qquad \qquad \qquad 
+\sum_{j=2}^{J-1}(\tilde a_j^{l+}(\xi_i-x)+\tilde b_j^{l+})\mathds{1}_{(0,\infty)\cap(y_j,y_{j+1}]}(\xi_i-x) \Big)
\int_{1-q_{i}}^{1-q_{i-1}}\mathds{1}_{[t_l,t_{l+1})}(t)dt\right\}\nonumber \\
&&\qquad  {\rm s.t.}\qquad (\ref{robust_Wra})-(\ref{robust_Wrj}),\;
 (\ref{robust_v_ob-LP})-(\ref{robust_v_ok-LP}), 
\edeqn
\item[Step 2.]
Solve the following minimization problem:
\bgeqn
\label{eq:two-stage-shortfall-solve}
\displaystyle \inf_{x\in \R} \{ x :h(x)\leq 0\}.
\edeqn
\end{itemize}
\end{algorithm}

Note that when the 
random variable $\xi$ is essentially 
bounded, we may solve problem (\ref{eq:two-stage-shortfall-solve})
by bisection method. 
Consider, for instance, that $\Xi=[0,0.5]$. 
We begin with an initial feasible solution $x_0$ and 
set $x_1 :=\frac{x_0}{2}$.
If $x_1$ is feasible,
then update the feasible set by $[0, x_1]$, 
otherwise set $[x_1, x_0]$, repeat the process.
We will come back to this in 
Section \ref{sec5:case_study}.

\section{Numerical tests}
\label{sec5:case_study}

To examine the performance of the PRGSR-CPT model and the computational schemes, we carry out some numerical tests
{\color{black} on problem 
(\ref{eq:robust-V-W-computation-SAA})
with an academic example where the true value function and weighting functions are known and how elicited 
preference information may affect the convergence of the worst case value function and weighting functions as more and more information is available. 
Since this is artificially made example, we use the well known random utility split approach \cite{ArD15} to generate pairwise comparison questionnaires (prospects) and then use
the GSR-CPT based on the tuple of true value function and weighting functions to select a preferred prospect (act like the DM). The preference information is then used to
construct the ambiguity set of the value functions.}

All of the numerical experiments are carried out with Matlab 2020b installed on a PC (16GB RAM, CPU 2.3 GHz) with Intel Core i7 processor. In this section, we report the details of the test and the outcomes.

\subsection{Setup}
Let $\xi$ be a random variable which is uniformly distributed 
over a discrete set 
$ 
\Xi=\{\xi_1,\cdots,\xi_N\}$ with $N=10$.
We sort out $\{\xi_i\}_{i=1}^N$ in non-decreasing order,
$\xi_1<\xi_2<\cdots<\xi_N$.
Specifically let 
\bgeqn 
\label{eq:Xi-test} 
\Xi &=&\{\xi_1,\cdots,\xi_{10}\}\nonumber\\
&=&\{0.4074,0.4529,0.0635,0.4567,0.3162,0.0488,0.1392,0.2734,0.4788,0.4824\}.
\edeqn 
We consider 
a specific situation where the true value function 
is 
\bgeq
v_*(x)=\left\{ \begin{array}{ll}
v_*^{+}(x)=x^{1/3}& \text{for} \; x\geq 0,\\
v_*^-(x)=-1.5(-x)^{0.2} & \text{for} \;x<0,
\end{array}
\right.
\edeq
and the true weighting functions $w_*^-(p)=w_*^+(p)$ are piecewise functions with $T$ breakpoints of $w(p)$:
$$
w(p)=\frac{p^{\gamma}}{p^{\gamma}+(1-p)^{\gamma}}, \;\gamma=0.6. 
$$
It is easy to verify that $1-w(1-p)=w(p)$.
Then we consider the case that $\bbe_{w^-w^+}[v(\xi)]=\int_{0}^{0.5} v(t)dw^-(F_{\xi}(t))$
and $1-w(1-p)=w(p)$.
We consider problem (PRGSR-CPT-S-Dis)
defined in (\ref{eq:robust-V-W-computation-SAA}):
\bgeqn
 \label{eq:robust-V-W-computation-SAA-experiment}
\rho_{{\cal V}_{El} \times {\cal W}_r}(\xi)
=\inf\limits_{x\in [0,0.5]} \left\{ x:\;  \sup_{(v,w^-)\in {\cal V}_{El}\times {\cal W}_r} \sum_{i=-m}^n \pi_i v(\xi_i-x) \leq 0\right\},
\edeqn
where
$v(x-\xi_i)=v^-(x-\xi_i)$ for $i=-m,\cdots,-1$,
$v(x-\xi_i)=v^+(x-\xi_i)$ for $i=0,\cdots,n$,
\bgeq
\pi_i &:=&\left\{ 
\begin{array}{ll}
 \displaystyle 
 w^-\left(\frac{i+m+1}{10}\right)-w^-\left(\frac{i+m}{10}\right)    & \inmat{for} \; 
i=-m,\cdots,-1
\\
  \displaystyle 
w^+\left(\frac{n-i+1}{10}\right)-w^+\left(\frac{n-i}{10}\right)    & \inmat{for} \; 
i=0,1,\cdots, n,
\end{array}
\right.
\edeq
and $m+n+1=10$,
 ${\cal V}_{El}={\cal V}_{ce}^K\cap {\cal V}_{pair}^{M}\cap {\cal V}_S$ and ${\cal W}_{r}=B(w_0,r)$. In this experiment, we assume the shape of the true value function is known, that is, it is S-shaped with 
 $v(0)=0$. Based on this information and the concrete positive values of 
 $\xi_i$ in the support set $\Xi$, we deduce that
$\rho_{(v,w^-,w^+)}(\xi)\in [0,0.5]$ for any value function and weighting functions.
This explains why we write the feasible set of $x$ as $[0,0.5]$ instead of $\R$ in problem~(\ref{eq:robust-V-W-computation-SAA-experiment}).

In the next subsections, 
we will describe in detail how the ambiguity sets ${\cal W}_{r}$ and ${\cal V}_{El}$
are constructed.

\subsection{Construction of ${\cal W}_r$}

As discussed earlier, we use a ball centered at a nominal weighting function to define
the ambiguity set ${\cal W}_r$ under the pseudo-metric,
 see Definition \ref{defi-Ball-N}.
 We take a specific $\tilde g(t)=1$ in the test, i.e., ${\cal W}_r := B_{L_1}(w_0,r)$, see (\ref{ambiguit_W2}).
 It follows by Theorem \ref{thm:reformulation-constraint-function-VW} that $$
 \sup_{(v,w^-)\in {\cal V}_{El}\times {\cal W}_r} \sum_{i=-m}^n \pi_i v(\xi_i-x)=\sup_{(v,w^-)\in {\cal V}_{El}\times \widetilde B_{L_1}(w_0,r)}\sum_{i=-m}^n \pi_i v(\xi_i-x),
 $$
 where
\bgeqn
\label{example_W}
\widetilde B_{L_1}(w_0,r)
=\left\{
w \in \mathscr W_T\;:
\begin{array}{l}
 \mbox{$w$ is increasing and inverse-$S$ shaped}\\
 w(0)=0, w(1)=1, \dd_1(w,w_0)\leq r
\end{array}
\right\},
\edeqn
where $\dd_1(w,w_0):=\int_{0}^{1}|\psi(t)-\psi^0(t)|dt$.
Define the following set of derivatives of $w\in{\widetilde B}_{L_1}(w_0,r)$,
\bgeqn
\label{ambiguity_pi11}
\widetilde \Psi_{L_1}(\psi^0,r)
:=\left\{\psi\in \Psi:
w\in {\widetilde B}_{L_1}(w_0,r)
\right\}
=\left \{\psi\in \Psi_T:
\sum_{l=1}^T|\psi_l-\psi^0_{l}|(t_{l+1}-t_l)\leq r
\right\},
\edeqn
where $\Psi$ 
is the set of  derivative functions of the weighting functions in ${\mathscr W}_{inS}$.

\subsection{Construction of ${\cal V}_{El}$}
\label{section-construct-V}
We randomly generate the certainty equivalent questionnaires and pairwise comparison questionnaires.
These questions  may be phrased as follows. 

\begin{itemize}
\item Q1. 
 For each pair of random prospects
$(\gamma_m,\eta_m)$, $m=1,\cdots,M$, which prospect does the DM prefer? As we discussed earlier in Section \ref{sec3:ambiguity_VW},
we represent by $\bbe_{w^-w^+}[v(\gamma_m)]\geq \bbe_{w^-w^+}[v(\eta_m)]$ if random prospect 
$\gamma_m$ is preferred to random prospect $\eta_m$.
Note that in this case, the DM's true weighting functions $w_*^-,w_*^+$
are known.

\item Q2.
What's the smallest amount of cash,
denoted by $a_k^{+}$,
that the DM would decline to commit instead of being exposed to 
the 
random 
prospect $A_k$
and what is the largest amount of cash,
denoted by $a_k^-$,
that DM would be willing to commit instead of being exposed to the random prospect
$A_k$ for $k=1,\cdots,K$.
\end{itemize}
 It remains to explain how the random prospects are generated. 
For Q1, we use the basic idea of the utility split approach
(see e.g. \cite{ArD15,DGX18}) 
to generate pairs of random prospects
$(\gamma_m,\eta_m)$, $m=1,\cdots,M$.
First, we explain how to generate random prospects.
Let
\bgeq
Z_1=\left\{
\begin{array}{ll}
r_1 & \mbox{with probability } 1-p,\\
r_3 & \mbox{with probability } p,
\end{array}
\right.
\mbox{and }
Z_2=r_2,
\edeq
where $r_1$ and $r_3$ are generated randomly over $[-0.5,0.5]$ with 
$r_1<r_3$, let $r_2=(r_1+r_3)/2$. Next, we need to set an appropriate value for $p$. 
Observe that 
\bgeq
\bbe_{w_*^-w_*^+}[v(Z_1)] =
(1-w_*^+(p))v(r_1)
+w_*^+(p)v(r_3)
\quad \text{and} \quad
\bbe_{w_*^-w_*^+}[v(Z_2)] =
v\left(\frac{r_1+r_3}{2}\right).
\edeq
Note that here the true weighting functions are known whereas 
the value function $v$ is unknown. Consider the case
\bgeqn
\label{eq:Z1-Z2-compaire}
(1-w_*^+(p))v(r_1)
+w_*^+(p)v(r_3) <
v\left(\frac{r_1+r_3}{2}\right).
\edeqn
Let
$\tilde{v}(x) = \frac{v(x)-v(r_1)}{v(r_3)-v(r_1)}$ for $x\in [r_1,r_3]$. 
Then  $\tilde v(r_1)=0$, $\tilde v(r_3)=1$ and consequently
(\ref{eq:Z1-Z2-compaire}) can be written as  
\bgeqn
\label{eq:tildev-Z1-Z2}
w_*^+(p)<
\tilde v\left(\frac{r_1+r_3}{2}\right).
\edeqn
In other words, the characterization of the value function 
$v$ is down to the characterization of the normalized value function $\tilde{v}$
at point $r_2$.
Let ${\cal V}^k_{pair}$ denote the set of value functions which are consistent to pairwise comparison questionnaires that we have already generated. 
Define 
 \bgeqn
 I_1:=\min_{v\in {\cal V}^k_{pair}\cap {\cal V}_S}\tilde v(r_2)
 \quad \text{and} \quad I_2:=\max_{v\in {\cal V}^k_{pair}\cap {\cal V}_S}\tilde v(r_2).
\edeqn
Since $\tilde{v}(r_2)\in [0,1]$, then $I_1, I_2\in [0,1]$.  We set the value of $p$ such that  
\bgeqn
w_*^+(p)=\frac{I_1+I_2}{2},
\edeqn
    and then 
use the true value function $v_*$
(acting as the DM)
to check  whether
  $\bbe_{w_*^-w_*^+}[\tilde v_*(Z_1)]<\bbe_{w_*^-w_*^+}[\tilde v_*(Z_2)]$ or not,
  that is,
   $w_*^+(p)<\tilde v_*(r_2)$ or not.
   Let
\bgeq
    \phi_{k+1}(y)
    &=&
     - \Big(w_*^-((1-p_{k+1})\mathds{1}_{y\geq r_1}+p_{k+1}\mathds{1}_{y\geq r_3})-w_*^-(\mathds{1}_{y\geq r_2})\Big) \nonumber \\
    &=& 
-\mathds{1}_{y\geq r_3}+\mathds{1}_{y\geq r_2 } - (1-w_*^+(p_{k+1}))\mathds{1}_{y\geq r_1 \& y< r_3},
\edeq
and $S$ be the set of all $r_1,r_2,r_3$ generated from $M$ questionnaires and points $-0.5$, $0.5$, $0$, and $\widetilde M=3M+3$. 
The following algorithm describes the procedures  for constructing ${\cal V}^M_{pair}\cap {\cal V}_S$.

\begin{algorithm}
\label{algorithm-pairwise-comparison}
Initialization: set $k :=0$ and the set of questionnaires is empty.

\begin{itemize}

\item[Step 1.] Choose $r_1$ and $r_3$ 
randomly from $[-0.5,0.5]$ 
with equal probability with $r_1<r_3$ 
and then set
    $r_2 :=\frac{r_1+r_3}{2}$.
    
\item[Step 2.] 
     Let $j_0 :=\max\{j\in \{1,\cdots,\widetilde M\}: t_j= 0\}$ and solve the following problem:
 \begin{subequations}
 \label{problem:generate-p}
\begin{align}
    \min\limits_{a,b}~&
\left(    (a_{1}r_2+b_{1}
   )\mathds{1}_{[t_1,t_2]}(r_2)+\sum_{j=2}^{\widetilde M-1}(a_{j}r_2+b_{j})
   \mathds{1}_{(t_j,t_{j+1}]}(r_2)+1.5\cdot(0.5)^{0.2}\right)/((0.5)^{1/3}+1.5\cdot(0.5)^{0.2})\\
   {\rm s.t.}~& a_{j-1}y_j+b_{j-1}=a_jy_j+b_j, \forall j=2,\cdots,\widetilde M-1,\\
  & a_1r_1+b_1=-1.5\cdot(0.5)^{0.2},\\
  & a_{\widetilde M-1}r_3+b_{\widetilde M-1}=(0.5)^{1/3},\\
  & \sum_{j=1}^{\widetilde M-1}a_{j} \int_{y_j}^{y_{j+1}}\phi_l(y) dy \leq 0, l=1,\cdots, k,\\
  & a_j\leq a_{j+1}, j=1, \cdots, j_{0}-1,\\
  & a_{j+1}\leq a_j, j=j_{0},\cdots, \widetilde M-2,\\
  & a_j\geq 0, j=1,\cdots, \widetilde M-1.
\end{align}
\end{subequations}
Let $I_1$ be the optimal value of the program. 
Replace the ``$\min$'' with ``$\max$'' and solve the maximization problem. 
Let $I_2$ be the optimal value of 
the maximization problem.
   Let 
   $p_{k+1}$ be such that
   $w_*^+(p_{k+1})=\frac{I_1+I_2}{2}$,
    and then use the true value function $v_*$ to check whether
\bgeqn
\bbe_{w_*^-w_*^+}[\tilde  v_*(Z_2)]>\bbe_{w_*^-w_*^+}[\tilde v_*(Z_1)]
\edeqn
or not,
  that is,
  $\tilde v_*(r_2)> w_*^+(p_{k+1})$ or not.
   If $\tilde v_*(r_2)> w_*^+(p_{k+1})$,
    then $Z_2$ is preferred over $Z_1$.
   Let
     \bgeqn
     \label{eq:defi_phim}
    \phi_{k+1}(y)=\left\{\begin{array}{cl}
    -1+w_*^+(p_{k+1})& \inmat{ if } r_1\leq y<r_2,\\
    w_*^+(p_{k+1}) & \inmat{ if } r_2\leq y<r_3,\\
    0 & \inmat{otherwise}.\\
    \end{array}
    \right.
    \edeqn
 Add the inequality  $\int_{-0.5}^{0.5} \phi_{k+1}(y) dv(y)\leq 0$ to the constraints of problem (\ref{problem:generate-p}).
    Otherwise,
   $\tilde v(r_2)\leq  w_*^+(p_{k+1})$,
      $Z_1$ is preferred over $Z_2$
    and add   inequality  $\int_{-0.5}^{0.5} -\phi_{k+1}(y) dv(y)\leq 0$ to the constraints of problem (\ref{problem:generate-p}).\\
\item[Step 3.]
Set $k:=k+1$ and 
return to Step 1.
\end{itemize}
\end{algorithm} 
For Q2,
we let
\bgeq
A_k :=\left\{
\begin{array}{ll}
r_k &\inmat{ with probability } p_k,\\
0 & \inmat{ with probability } 1-p_k,
\end{array}
\right.
\edeq
where $r_k$ is randomly 
generated 
with uniform distribution over $[0,0.5]$ 
and $p_k$ is also randomly generated
with the uniform distribution over $\{0.1,0.2,\cdots,0.9\}$.
The randomness 
will enable us to pick up any value in the domain 
of the value function and with any likelihood. 
The upper and lowers bounds $a_k^{+},a_k^{-}$ 
are set as follows:
\bgeqn
a_k^{+}=(1+\tau)\rho_{(v_*,w_*^{-},w_*^{+})}(A_k),\;\;
a_k^{-}=(1-\tau)\rho_{(v_*,w_*^{-},w_*^{+})}(A_k),
\label{eq:a_k+-}
\edeqn
where $\rho_{(v_*,w_*^{-},w_*^{+})}(A_k)$ is the true
GSR-CPT value of $A_k$, i.e., 
\bgeq
\rho_{(v_*,w_*^{-},w_*^{+})}(A_k):=
\inf\limits_{x\in [0,r_k]}  \left\{x: (r_k-x)^{1/3}w_*^{+}(p_k) 
-1.5x^{0.2}(1-w_*^+(p_k) 
\leq 0 \right\},
\edeq
and 
$\tau$ is randomly generated with a uniform distribution over $[0,0.05]$.

The following algorithm describes the procedures for constructing ${\cal V}^k_{ce}$.
\begin{algorithm}
Initialization: set $k :=0$ and  the set of questionnaires is empty. 
\begin{itemize}
 \item[Step1.]
Choose randomly $r_k\in [0,0.5]$ and $p_k\in \{0.1,0.2,\cdots,0.9\}$.
			
\item[Step2.] Compute $a^-_k,a^+_k$ via formulae (\ref{eq:a_k+-}) and construct
		    $${\cal V}_{ce}^k=\{v\in {\mathscr V}: \bbe_{w^-_*w^+_*}[v(A_i-a^+_i)]\leq 0, \bbe_{w^-_*w^+_*}[v(A_i-a^-_i)]\geq 0,i=1,\cdots,k\}.
		    $$
\item[Step3.]
	      If problem $\rho_{{\cal V}_{pair}^{M}\cap {\cal V}_{ce}^i\times {\cal W}_r}(\xi)$ in (\ref{eq:robust-V-W-computation-SAA-experiment}) is feasible, then set $k=k+1$. The loop goes to step 1.
\end{itemize}
\end{algorithm}

\subsection{Numerical experiments}

We consider Problem (\ref{eq:robust-V-W-computation-SAA-experiment}) and use
Algorithm~\ref{algorithm:thm4-1}
to solve it. Since $\tilde g(t)=1$,
we can rewrite (\ref{eq:two-stage-reformulate-tildephi}) as
\bgeqn
\label{eq:two-stage-reformulate-experiment}
 h(x)=
 && \inf_{\tilde a^-,\tilde b^-,\tilde \psi^-} \left\{\sum_{i=-m}^{n} \sum_{l=1}^T
\Big( (\tilde a_1^{l-}(\xi_i-x)+\tilde b_1^{l-})\mathds{1}_{[y_1,y_2]}(\xi_i-x) \right. \nonumber \\
&& \qquad \qquad \qquad \qquad \left. +\sum_{j=2}^{J-1}(\tilde a_j^{l-}(\xi_i-x)+\tilde b_j^{l-})\mathds{1}_{(y_j,y_{j+1}]}(\xi_i-x) \Big)
\int_{q_{i-1}}^{q_{i}}\mathds{1}_{[t_l,t_{l+1})}(t)dt\right\} \nonumber \\
&&\qquad {\rm s.t. } ~~\qquad (\ref{robust_Wra}), (\ref{robust_Wrc}),  (\ref{robust_Wre}), \nonumber \\
&&\qquad \qquad \qquad 
\sum_{l=1}^T|\psi_l^--\psi^0_{l}|(t_{l+1}-t_l)\leq r,
\nonumber \\
&& \qquad \qquad \qquad \psi_l^->0, l=1,\cdots, T, \nonumber \\
&& \qquad \qquad \qquad \sum_{j=1}^{J-1}  \tilde a_j^{l-} \int_{y_j}^{y_{j+1}} \phi_m(y)dy\leq 0, m=1,\cdots, M,\nonumber \\
&& \qquad \qquad \qquad 
(\ref{robust_v_ob-LP}),(\ref{robust_v_oc-LP}),(\ref{robust_v_od-LP}),(\ref{robust_v_oe-LP}),(\ref{robust_v_of-LP}),(\ref{robust_v_og-LP}),(\ref{robust_v_oh-LP}),(\ref{robust_v_oj-LP}), \nonumber \\
&&  \qquad \qquad \qquad  a_j^{l-}>0, a_j^{l+}>0, j=1,\cdots,J, l=1,\cdots,T,
\edeqn
 where $\tilde \psi^-:=(\psi_1^-,\cdots,\psi_T^-)$, 
 $\tilde \psi^0:=(\psi^0_{1},\cdots,\psi^0_{T})$,
  $q_k^{0}=0$,
 $q_k^{1}=1-p_k$,
$q_k^{2}=1$, 
$A_k^1=0,A_k^2=r_k$,
\bgeq
\hat c^{l-}_k&=& \sum_{i=1}^{2} 
\Big( (\tilde a_1^{l-}(A_k^i-a_k^{+})+\tilde b_1^{l-})\mathds{1}_{[y_1,y_2]}(A_k^i-a_k^{+}) \\
&& \quad~ +\sum_{j=2}^{J-1}(\tilde a_j^{l-}(A_k^i-a_k^{+})+\tilde b_j^{l-})\mathds{1}_{(y_j,y_{j+1}]}(A_k^i-a_k^{+}) \Big)
\left(w_*^-\left(q^i_k\right)-w_*^-\left(q^{i-1}_k\right)\right), 
\edeq
\bgeq
\tilde c^{l-}_k &=&\sum_{i=1}^{2} 
\Big( (\tilde a_1^{l-}(A_k^i-a_k^{-})+\tilde b_1^{l-})\mathds{1}_{[y_1,y_2]}(A_k^i-a_k^{-}) \\
&& \quad~ +\sum_{j=2}^{J-1}(\tilde a_j^{l-}(A_k^i-a_k^{-})+\tilde b_j^{l-})\mathds{1}_{(y_j,y_{j+1}]}(A_k^i-a_k^{-}) \Big)
\left(w_*^-\left(q^i_k\right)-w_*^-\left(q^{i-1}_k\right)\right).
\edeq
In this experiment,
we examine how increase of the
number of the pairwise comparisons $M$ affects the worst 
case value function.
We also investigate the change of the worst case weighting function as the radius of the $\tilde g$-ball is driven to zero. 
The breakpoints of the piecewise linear weighting function 
are chosen with  $t_i=0.1i$ for $i=1,\cdots,10$ and $T=9$.
The cdf function of $\xi$ 
takes a value of $q_k=k/10$ at $\xi_k$ for  $k=1,\cdots,10$ with $q_0=0$, see (\ref{eq:Xi-test}) for the definition of support set of $\xi$. 
Note that the optimal value of the true problem is $0.2044$.

\noindent\underline{\bf Impact of the pairwise comparisons and certainty equivalents}

Figure \ref{fig:experiment-value-function} (a) depicts 
changes of the worst-case value functions
when the number of pairwise comparisons $M$ increases
from $5$ to $20,40$ and $100$. The number of certainty equivalents $K$ is fixed at $5$. The radius of the $\tilde g$-ball is also fixed at $r=0.01$.
From the figure,
we can see the worst-case value function 
calculated with each $M\in \{5,20,40,100\}$
converges gradually 
to the true one 
$M$ increases.
This is because
more and more 
information of the DM's preference 
CPT 
is elicited via Algorithm~\ref{algorithm-pairwise-comparison}
and the ambiguity set 
${\cal V}^M_{pair}\cap {\cal V}_S$
shrinks. 
Figure \ref{fig:experiment-value-function} (b) depicts 
changes of the optimal values of problem (\ref{eq:robust-V-W-computation-SAA-experiment}) for $M\in \{5,20,40,60,80,100\}$

Figure \ref{fig:experiment-value-function_CE} (a) depicts 
changes of the worst-case value functions
when the number of certainty equivalents $K$ increases
from $5$ to $20,40$ and $100$. The number of pairwise comparisons $M$ is fixed at $5$. The radius of the $\tilde g$-ball is also fixed at $r=0.01$.
From the figure,
we can see the worst-case value function 
calculated with each $K\in \{5,20,40,100\}$
converges gradually 
to the true one 
as
$K$ increases.
This is because
more and more 
information of the DM's preference 
CPT 
is elicited via Algorithm~\ref{algorithm-pairwise-comparison}
and the ambiguity set 
${\cal V}^K_{ce}\cap {\cal V}_S$
shrinks. 
Figure \ref{fig:experiment-value-function} (b) depicts 
changes of the optimal values of problem (\ref{eq:robust-V-W-computation-SAA-experiment}) for $K\in \{5,20,40,60,80,100\}$.

\begin{figure}[!htbp]
\minipage{0.5\textwidth}
 \centering
  \includegraphics[width=\linewidth]{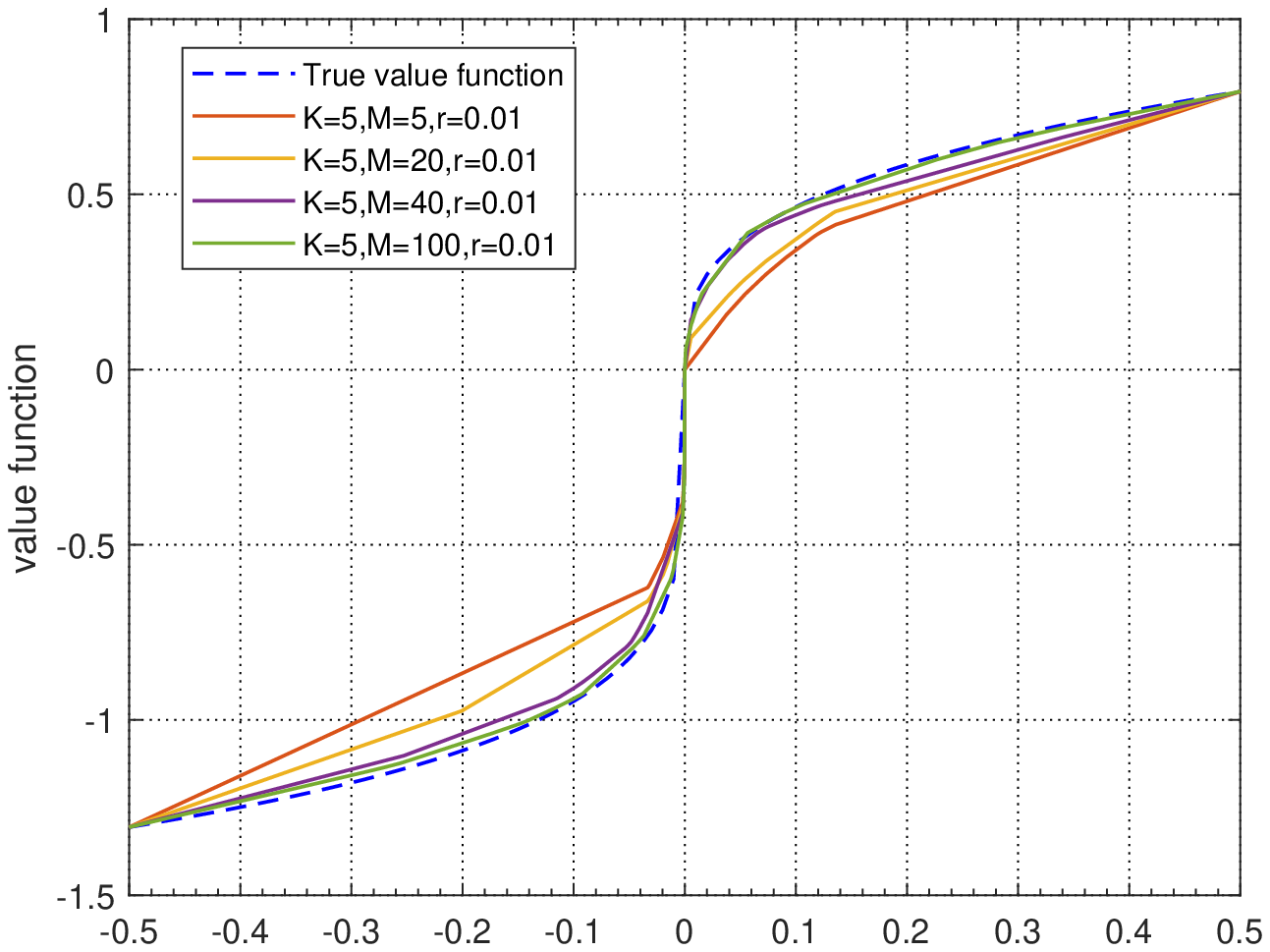}
  \textbf{(a)}
\endminipage\hfill
\minipage{0.5\textwidth}
  \centering
  \includegraphics[width=\linewidth]{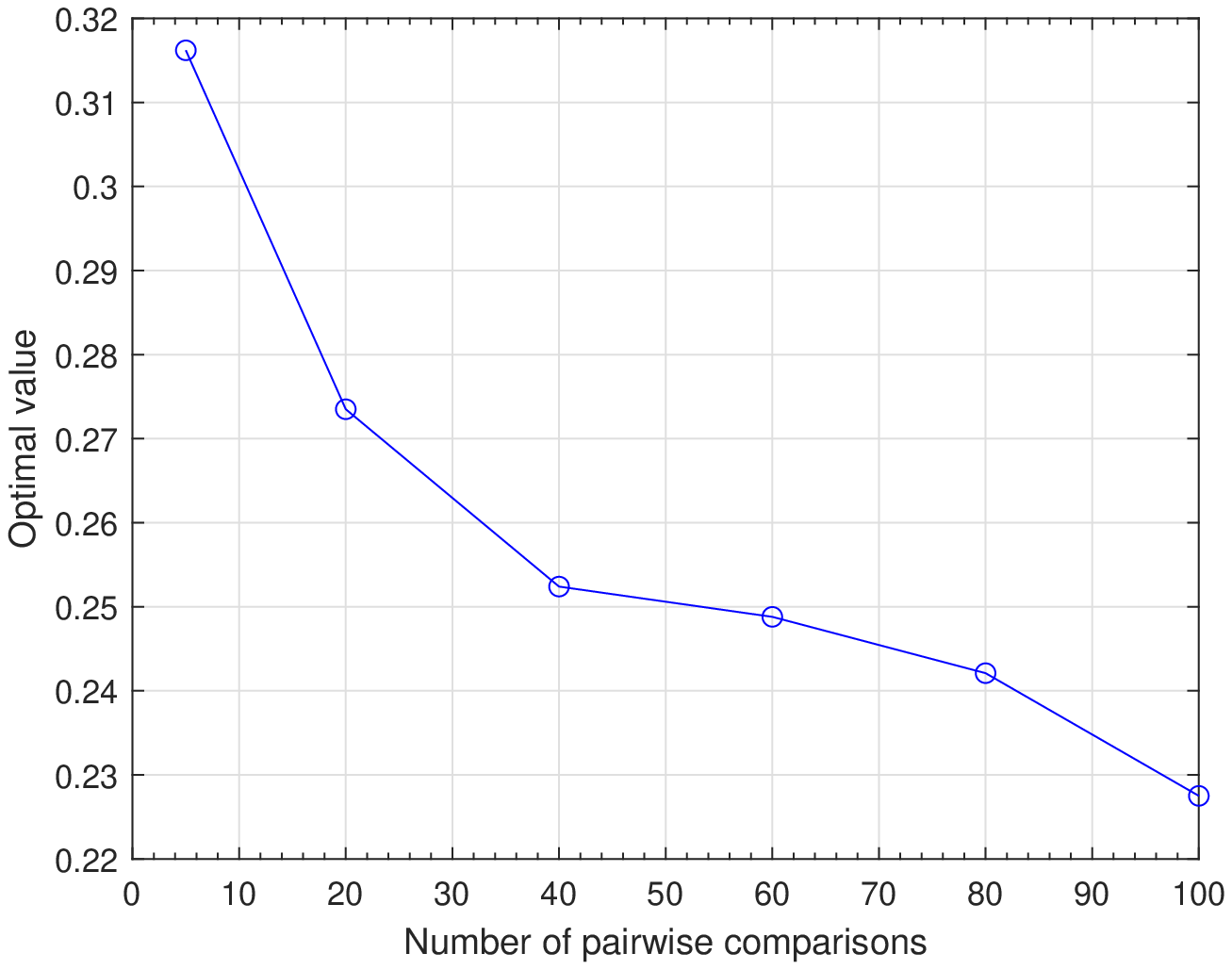}
  \textbf{(b)}
\endminipage
\caption{\small \textbf{(a)} Worst case value function vs 
increase of $M$.
\quad \textbf{(b)} 
Change of $\rho_{{\cal V}_{El}\times {\cal W}_r}(\xi)$ in problem (\ref{eq:robust-V-W-computation-SAA-experiment})
vs increase of $M$.}
\label{fig:experiment-value-function}
\end{figure}

\begin{figure}[!htbp]
\minipage{0.5\textwidth}
 \centering
  \includegraphics[width=\linewidth]{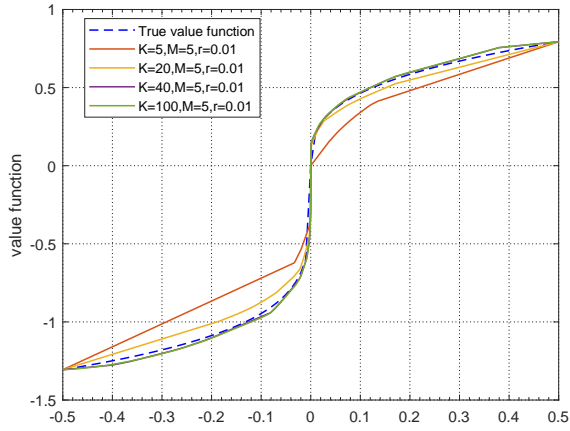}
  \textbf{(a)}
\endminipage\hfill
\minipage{0.5\textwidth}
  \centering
  \includegraphics[width=\linewidth]{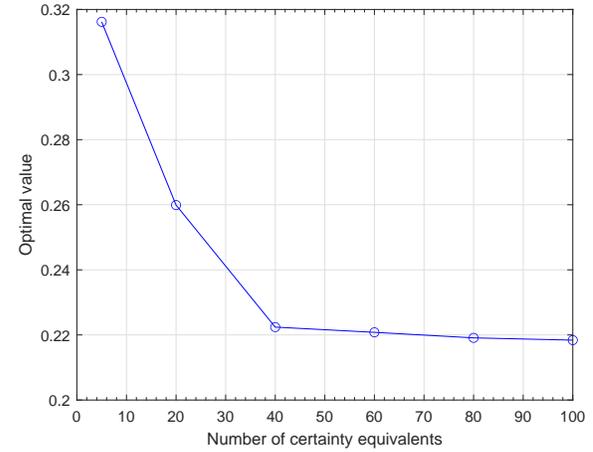}
  \textbf{(b)}
\endminipage
\caption{\small \textbf{(a)}Worst case value function vs 
increase of $K$.
\quad \textbf{(b)} 
Change of $\rho_{{\cal V}_{El}\times {\cal W}_r}(\xi)$ in problem (\ref{eq:robust-V-W-computation-SAA-experiment})
vs increase of $K$.}
\label{fig:experiment-value-function_CE}
\end{figure}

\noindent\underline{\bf Impact of ambiguity of the weighting functions}

Figure \ref{fig:experiment-weighting-function} (a) depicts changes of the worst-case weighting function w.r.t. decrease
of the radius while the number of pairwise comparisons is fixed at $M=5$ and the number of certainty equivalents is fixed at $K=5$.
From the figure,
we can see the convergence of the worst-case weighting functions to the true 
as $r$ reduces to $0$.

Figure \ref{fig:experiment-weighting-function} (b) displays the optimal values
of PRGSR-CPT for the ambiguity sets $({\cal V}_{El},{\cal W}_r)$.
We can see that for the fixed number of pairwise comparisons,
 the PRGSR-CPT becomes smaller as the radius of ambiguity set decreases.

\begin{figure}[!htbp]
\minipage{0.5\textwidth}
 \centering
  \includegraphics[width=\linewidth]{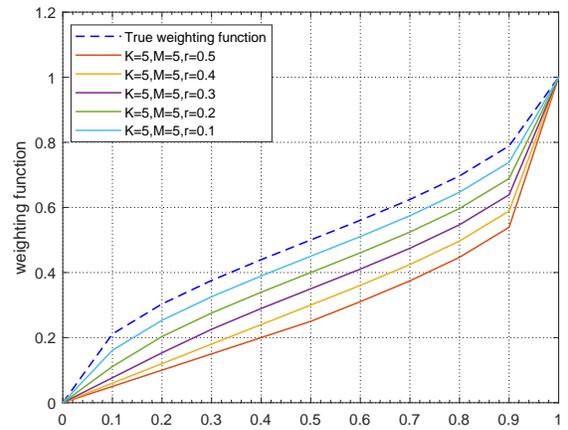}
  \textbf{(a)}
\endminipage\hfill
\minipage{0.5\textwidth}
  \centering
  \includegraphics[width=\linewidth]{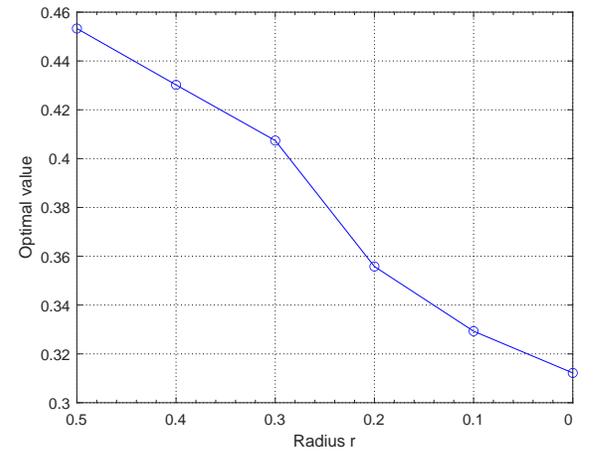}
  \textbf{(b)}
\endminipage
\caption{\small \textbf{(a)}
Worst case weighting function vs
decrease of $r$.
\quad \textbf{(b)} 
Change of $\rho_{{\cal V}_{El}\times {\cal W}_r}(\xi)$ in problem (\ref{eq:robust-V-W-computation-SAA-experiment})
vs decrease of $r$.}
\label{fig:experiment-weighting-function}
\end{figure}

The test results are consistent with our expectation that 
as more and more information about the DM's preferences is obtained, 
the ambiguity of the true value function and weighting function reduces, the preference robust
generalized shortfall risk converges to the true GSR-CPT.

\section{Concluding remarks}

In this paper, we explore Mao and Cai's generalized shortfall risk measure 
when the value function and/or the weighting function are ambiguous. We introduce the concept of 
preference robust generalized shortfall risk measure which is based on a pair of the worst value function and 
weighting function from a joint  ambiguity set value/weighting functions, investigate the properties of the robust generalized shortfall risk measure and develop computational algorithms for calculating the robust risk.
The numerical tests show that the algorithms perform well w.r.t. increase of information about the DM's risk preferences.

There are some limitations in the developed models and preference elicitation approaches. One  is that the ambiguity set of the weighting functions is constructed through a ball centered at a nominal weighting function, it might be interesting to combine it with the pairwise comparison approach to reduce the ambiguity. This approach is expected to work well for each fixed value function but it seems to be challenging to elicit both the value function and the weighting function simultaneously. The other is that we are unable to use the CPT-based shortfall risk measure to represent the DM’s
preference choice, rather we use the distorted expected values of prospects, that is, $\bbe[w_-w_+[v(\cdot)]$ to
describe the DM’s pairwise choice. As discussed in Section 3.2, there could be a gap between the two approaches. 
Note also that in the numerical tests, instances of pairwise comparisons and certainty equivalents 
are generated randomly without any costs. In practice, eliciting a DM's preferences in this manner may 
incur financial costs, see Vayanos et al.~\cite{VYMDR20}.  
Thus it might be interesting to explore the idea as to how to optimally design questionnaires to 
elicit the true value/weighting functions  using approaches such as \cite{BeO13}. 
We leave all these for future research.


\appendix
  
 \renewcommand\thesection{Appendix~\Alph{section}}

\setcounter{section}{0}
\section{Supplementary materials}

\renewcommand\thesection{\Alph{section}}
\subsection{Properties of preference functional $\bbe_{w^-w^+}[v(\cdot)]$}
\begin{lemma}
\label{lemma:monotone-of-bbe-xieta}
Let $\bbe_{w^-w^+}[v(\zeta)]$ be defined as in (\ref{eq:defi-bbe-wwv}). Then
    \bgeqn
  \bbe_{w^-w^+}[v(\eta)]\leq \bbe_{w^-w^+}[v(\xi)]\,\, \forall \eta,\xi \in {\cal L}^0(\R) \inmat{ with } \eta(\omega)\leq \xi(\omega), \forall \omega \in \Omega.
    \edeqn
\end{lemma}
\begin{proof} 
By definition (see (\ref{eq:defi-bbe-wwv})),
\bgeq
 \bbe_{w^-w^+}[v(\xi)]
 =\int_{0}^{\infty}w^+(\mathbb P(v(\xi)>t))dt
 +\int_{0}^{\infty} -w^-(\mathbb P(-v(\xi)>t))dt.
\edeq
Since
$\eta(\omega)\leq \xi(\omega)$ for all $\omega\in \Omega$ and $v(\cdot)$ is monotonically increasing, then
$v(\eta(\omega))\leq v(\xi(\omega))$ for all $\omega\in \Omega$,
and consequently 
\bgeq
\mathbb P(v(\eta)>t)\leq \mathbb P(v(\xi)>t) \quad \inmat{and} \quad \mathbb P(-v(\eta)>t)\geq \mathbb P(-v(\xi)>t).
\edeq
Moreover, since $w^-,w^+$ are strictly increasing,
we have 
\bgeq
w^+(\mathbb P(v(\eta)>t))\leq w^+(\mathbb P(v(\xi)>t))\quad \text{and}\quad -w^-(\mathbb P(-v(\eta)>t))\leq -w^-(\mathbb P(-v(\xi)>t)).
\edeq
A combination of the inequalities give rise to 
$$
\bbe_{w^-w^+}[v(\eta)]\leq \bbe_{w^-w^+}[v(\xi)].
$$
The proof is complete.
\hfill $\Box$
\end{proof}

\begin{lemma}
\label{lemma:monotone-of-bbewv}
Let $\bbe_{w^-w^+}[v(\cdot)]$ be defined as in (\ref{eq:defi-bbe-wwv}) and $\xi\in {\cal L}^0(\R)$.
Then the following assertions hold.
\begin{itemize}
    \item[(i)] 
    For any 
    $a, b\in \R$ with $a\leq b$,
    \bgeqn
  \bbe_{w^-w^+}[v(\xi-a)]\geq \bbe_{w^-w^+}[v(\xi-b)].
    \edeqn
    \item[(ii)]
   For any 
    $a, b\in \R$ with $a< b$,
      \bgeqn
    \bbe_{w^-w^+}[v(\xi-a)]>
    \bbe_{w^-w^+}[v(\xi-b)]\,\, \forall a < b.
    \edeqn
 \end{itemize}
\end{lemma}
\begin{proof} 
Part (i) follows from directly from 
 Lemma~\ref{lemma:monotone-of-bbe-xieta}.

Part(ii). Let 
$\Omega_1:=\{\omega\in \Omega: \xi(\omega)\leq b\}$.
For real number $a\in \R$,
let $(a)_+:=\max\{a,0\}$
and $(a)_-:=\min\{a,0\}$.
For any $a< b$,
we have 
\bgeq
\xi(\omega)-b< [\xi(\omega)-a]_-,  \forall \omega\in \Omega_1 \quad \inmat{and}\quad
\xi(\omega)-b< [\xi(\omega)-a]_+  \forall \omega\in \Omega\backslash{\Omega_1}.
\edeq
Since the value function $v\in \mathscr V$ is strictly increasing,  then
\bgeq
v^-(\xi(\omega-b))< v^-([\xi(\omega-a)]_-),  \forall \omega\in \Omega_1 \quad \inmat{and}\quad
v^+(\xi(\omega)-b)< v^+([\xi(\omega)-a]_+), \forall \omega\in \Omega\backslash{\Omega_1}.
\edeq
Let 
\bgeq
R_{1}:=\int_{\Omega_1} v^-(\xi(\omega)-b)\psi^-(\mathbb P(\{\tilde \omega\in \Omega:\xi(\tilde \omega)\leq \xi(\omega)\})) d \mathbb P(\omega)
\edeq
and 
\bgeq
R_{2}:=\int_{\Omega\backslash{\Omega_1}} v^+(\xi(\omega)-b)\psi^+(1-\mathbb P(\{\tilde \omega\in \Omega:\xi(\tilde \omega)\leq \xi(\omega)\})) d \mathbb P(\omega).
\edeq
By the definition of $\bbe_{(v,w^-,w^+)}[v(\cdot)]$, 
we have 
\bgeq
\bbe_{w^-w^+}[v(\xi-b)]
&=& \int_{\Omega_1} v^-(\xi(\omega)-b)\psi^-(\mathbb P(\{\tilde \omega\in \Omega:\xi(\tilde \omega)\leq \xi(\omega)\})) d \mathbb P(\omega)\\
&&+\int_{\Omega\backslash{\Omega_1}} v^+(\xi(\omega)-b)\psi^+(1-\mathbb P(\{\tilde \omega\in \Omega:\xi(\tilde \omega)\leq \xi(\omega)\})) d \mathbb P(\omega)\\
&=&R_1+R_2.
\edeq
Since we have either $\mathbb{P}(\Omega_1)>0$ or $\mathbb{P}(\Omega\backslash \Omega_1)>0$ or both, then
\bgeq
\bbe_{w^-w^+}[v(\xi-b)]
&< & \int_{\Omega_1} v^-([\xi(\omega)-a]_-)\psi^-(\mathbb P(\{\tilde \omega\in \Omega:\xi(\tilde \omega)\leq \xi(\omega)\})) d \mathbb P(\omega)\\
&&+\int_{\Omega\backslash{\Omega_1}} v^+([\xi(\omega)-a]_+)\psi^+(1-\mathbb P(\{\tilde \omega\in \Omega:\xi(\tilde \omega)\leq \xi(\omega)\})) d \mathbb P(\omega)\\
&=& \int^{b}_{-\infty} v^-([u-a]_-)\psi^-(F_{\xi}(u))dF_{\xi}(u) 
+\int^{\infty}_{b} v^+([u-a]_+)\psi^+(1-F_{\xi}(u))dF_{\xi}(u)  \\
&\leq &\int^{a}_{-\infty}
v^-(u-a)\psi^-(F_{\xi}(u))dF_{\xi}(u) 
+\int^{\infty}_{a} v^+(u-a)\psi^+(1-F_{\xi}(u))dF_{\xi}(u)\\
&=&\bbe_{w^-w^+}[v(\xi-a)].
\edeq
The proof is complete.
\hfill $\Box$
\end{proof}

\subsection{Proof of Theorem \ref{thm:reformulation-constraint-function-VW}}
\label{app:proof-theorem4-1}
\textbf{Proof.} We divide the proof into four steps. 

\underline{ \bf Step 1.}
By definition, 
we have 
$$
\sup_{(v,w^-,w^+)\in {\cal V}_{El}\times {\cal W}_r\times {\cal W}_r} \sum_{i=-m}^n \pi_i v(x-\xi_i)
=\sup_{v\in {\cal V}_{El}}\sup_{(w^-,w^+)\in {\cal W}_r\times {\cal W}_r} \sum_{i=-m}^n \pi_i v(x-\xi_i),
$$
where ${\cal V}_{El}={\cal V}_{ce}^K \cap {\cal V}^M_{pair}\cap {\cal V}_{S}$ and ${\cal W}_r=B(w_0,r)$.
For each fixed $v\in {\cal V}_{El}$, let
$$
\phi(v):=\sup_{(w^-,w^+)\in {\cal W}_r\times {\cal W}_r} \sum_{i=-m}^n \pi_i v(x-\xi_i).
$$
Then 
$\sup_{(v,w^-,w^+)\in {\cal V}_{El}\times {\cal W}_r\times {\cal W}_r} \sum_{i=-m}^n \pi_i v(x-\xi_i)=\sup_{v\in {\cal V}_{El}} \phi(v)$.
It suffices to derive the tractable reformulation of $\phi(v)$ for 
each $v\in {\cal V}_{El}$.

\underline{\bf Step 2.}
Let  $q_{-m-1} =0,$ $q_i =\sum_{j=-m}^{i}p_j$ for $i=-m,\cdots,n$.
We choose the set of breakpoints of the piecewise linear weighting functions such that it contains $q_i,1-q_i$, i.e.,
\bgeqn
\label{eq:breakpoints-q-t}
\{q_i: i=-m-1,\cdots,-1\}\cup \{1-q_i: -1,\cdots,n\} \subset \{t_1,\cdots,t_{T+1}\}. 
\edeqn
Define the 
$\tilde g$-ball of piecewise linear inverse 
$S$-shaped  weighting functions

centered at $w_0$ with radius $r$
\bgeqn
\label{ambiguity_W1}
{\widetilde B}(w_0,r):=\left\{w \in \mathscr W_T\cap {\mathscr W}_{inS}:
\dd_{\tilde g}(w,w_0)\leq r
\right\}.
\edeqn
Denote by  $\Psi(\psi^0,r)$ 
the set of all   derivative functions of $w\in B(w_0,r)$, i.e.,
\bgeq
\Psi(\psi^0,r)
:=\left\{\psi\in \Psi: w\in  B(w_0,r) \right\}
=\left \{\psi\in \Psi:
\displaystyle \sup_{g\in {\mathscr G}} \left|\int_{0}^{1}g(t)(\psi(t)-\psi^0(t))dt \right| \leq r
\right\},
\edeq
where $\Psi$ is the set of all derivative functions 
of $w\in {\mathscr W}_{inS}$ \footnote{Note that 
$ B(w_0,r)$ may consist of weighting functions which 
are not differentiable at some points of $[0,1]$.
Here we assume that  $w\in {\mathscr W}_{inS}$ is continuously differentiable over $[0,1]$ except a finite number of points
in which case the non-differentiable points do not affect the integral $\int_{0}^{1}g(t)(\psi(t)-\psi^0(t))dt$.
}.
Note that for $\psi\in \Psi(\psi^0,r)$, it is not necessarily 
a step-like function. Thus we propose to use step-like functions with jump points at set $\{t_2,\cdots,t_T\}$ 
to approximate it. Consequently we define  
  \bgeqn
 \label{eq: set-psi-T-derivative-weighting}
{\widetilde \Psi}(\psi^0,r)=\left \{\psi\in \Psi_T: 
\displaystyle \sup_{g\in {\mathscr G}} \left| \sum_{l=1}^T\int_{t_l}^{t_{l+1}}g(t)(\psi_l-\psi^0_{l})dt \right| \leq r
\right\},
\edeqn
where $\Psi_T$ is defined as in (\ref{defi:Psi-T}).
Since $\pi_i=w^-\left(q_i\right)-w^-\left(q_{i-1}\right),$ for $i=-m,\cdots,-1$, 
and $\pi_i=w^+\left(1-q_{i-1}\right)-w^+\left(1-q_i\right)$, for $i=0,1,\cdots, n$, 
we can write $\phi(v)$ as 
\bgeqn
\label{robust_W-BN-xi}
\phi(v)=\sup_{(\psi^-,\psi^+)\in \Psi(\psi^0,r)\times \Psi(\psi^0,r)} \left\{\sum_{i=-m}^{-1}v^{-}(\xi_i-x)\int_{q_{i-1}}^{q_{i}}\psi^-(t)dt+\sum_{i=0}^n v^{+}(\xi_i-x)\int_{1-q_{i}}^{1-q_{i-1}}\psi^+(t)dt\right \}.\;\;
\edeqn
 Next,  we will prove that $\phi(v)$ is equal to 
 $$
 \widetilde{\phi}(v):=\sup_{(\psi^-,\psi^+)\in {\widetilde \Psi}(\psi^0,r)\times {\widetilde \Psi}(\psi^0,r)} \left\{\sum_{i=-m}^{-1}v^{-}(\xi_i-x)\int_{q_{i-1}}^{q_{i}}\psi^-(t)dt+ \sum_{i=0}^n v^{+}(\xi_i-x)\int_{1-q_{i}}^{1-q_{i-1}}\psi^+(t)dt\right\}
 $$
 for any $v\in {\cal V}_{El}$.
 Since $ {\widetilde B}(w_0,r)\subset B(w_0,r)$,
 we have $\widetilde{\phi}(v) \leq \phi(v).$
 It suffices to show that $\widetilde{\phi}(v) \geq \phi(v)$.
 Let $\varepsilon>0$,
 and $w^{*-}_{\varepsilon},w^{*+}_{\varepsilon}\in B(w_0,r)$ be such that 
 \bgeqn
 \label{eq:w-sup-epsilon}
 &&\sum_{i=-m}^{-1} \left(w^{*-}_{\varepsilon}(q_{i})-w^{*-}_{\varepsilon}(q_{i-1})\right)v(\xi_i-x)
  +\sum_{i=0}^n\left( w^{*+}_{\varepsilon}(1-q_{i-1})-w^{*+}_{\varepsilon}(1-q_{i})\right) v(\xi_i-x) \nonumber \\
 &\geq& \sup_{(w^-,w^+)\in B(w_0,r)\times B(w_0,r)}\left\{\sum_{i=-m}^{-1} \left(w^{-} (q_{i})-w^{-}(q_{i-1})\right) v(\xi_i-x) \right.\nonumber \\
&& \left. +\sum_{i=0}^n\left( w^{+}(1-q_{i-1})-w^{+}(1-q_{i})\right) v(\xi_i-x)\right\}-\varepsilon.
 \edeqn
 We can find piecewise linear functions $w^{*-}_{\varepsilon, T}, w^{*+}_{\varepsilon, T}$ such that 
 $w^{*-}_{\varepsilon, T}(t_i)=w^{*-}_{\varepsilon}(t_i)$
 and $w^{*+}_{\varepsilon, T}(t_i)=w^{*+}_{\varepsilon}(t_i)$,
 $i=1,\cdots,T+1$,
$$
w^{*-}_{\varepsilon, T}(t)
:=w^{*-}_{\varepsilon}(t_{i-1})+\frac{w^{*-}_{\varepsilon}(t_i)-w^{*-}(t_{i-1})}{t_i-t_{i-1}}(t-t_{i-1}), \mbox{ for } t\in [t_{i-1},t_i], i=2,\cdots, T+1,
$$
$$
w^{*+}_{\varepsilon, T}(t)
:=w^{*+}_{\varepsilon}(t_{i-1})+\frac{w^{*+}_{\varepsilon}(t_i)-w^{*+}(t_{i-1})}{t_i-t_{i-1}}(t-t_{i-1}), \mbox{ for } t\in [t_{i-1},t_i], i=2,\cdots, T+1.
$$
For $w^{*-}_{\varepsilon},w^{*+}_{\varepsilon}\in B(w_0,r)$,
we have 
\bgeq
\dd_{\tilde g}(w^{*-}_{\varepsilon},w_0)=\sup_{g\in {\mathscr G}} \left|\int_{0}^{1} g(t)(\psi_{\varepsilon}^-(t)-\psi^0(t))dt \right|\leq r, \\
\dd_{\tilde g}(w^{*+}_{\varepsilon},w_0)=\sup_{g\in {\mathscr G}} \left|\int_{0}^{1} g(t)(\psi_{\varepsilon}^+(t)-\psi^0(t))dt \right| \leq r.
\edeq
Then 
\bgeq
&&\dd_{\tilde g}(w^{*-}_{\varepsilon,T},w_0)=\sup_{g\in {\mathscr G}} \left|\sum_{l=2}^{T+1}\int_{t_{l-1}}^{t_{l}} g(t)\left(\frac{w^{*-}_{\varepsilon}(t_l)-w^{*-}(t_{l-1})}{t_l-t_{l-1}}-\psi^0_l\right)dt \right|=\dd_{\tilde g}(w^{*-}_{\varepsilon},w_0)\leq r, \\
&&\dd_{\tilde g}(w^{*+}_{\varepsilon,T},w_0)
=\sup_{g\in {\mathscr G}} \left|\sum_{l=2}^{T+1}\int_{t_{l-1}}^{t_{l}} g(t)\left(\frac{w^{*+}_{\varepsilon}(t_l)-w^{*+}(t_{l-1})}{t_l-t_{l-1}}-\psi^0_l\right)dt \right|=
\dd_{\tilde g}(w^{*+}_{\varepsilon},w_0)\leq r,
\edeq
which means $w^{*-}_{\varepsilon}, w^{*+}_{\varepsilon}\in {\widetilde B}(w_0,r)$.
Since $q_i,1-q_i\in \{t_1,\cdots,t_{T+1}\}$ in (\ref{eq:breakpoints-q-t}), then
$w^{*-}_{\varepsilon, T}(q_i)=w^{*-}_{\varepsilon}(q_i)$, for $i=-m-1,\cdots,-1$, and $w^{*+}_{\varepsilon, T}(1-q_i)=w^{*+}_{\varepsilon}(1-q_i)$,
for $i=0,\cdots,n$,
and 
 \bgeq
 \widetilde{\phi}(v)&=& \sup_{(w^-,w^+)\in {\widetilde B}(w_0,r)\times {\widetilde B}(w_0,r)} \left\{\sum_{i=-m}^{-1} \left(w^{-}_{\varepsilon}(q_{i})-w^{-}_{\varepsilon}(q_{i-1})\right) v(\xi_i-x)\right.\\
 && \left.+\sum_{i=0}^n\left( w^{+}_{\varepsilon}(1-q_{i-1})-w^{+}_{\varepsilon}(1-q_{i})\right)v(\xi_i-x) \right\}\\ 
 &\geq & \sum_{i=-m}^{-1} \left(w^{*-}_{\varepsilon,T}(q_{i})-w^{*-}_{\varepsilon,T}(q_{i-1})\right) v(\xi_i-x) +\sum_{i=0}^n\left( w^{*+}_{\varepsilon,T}(1-q_{i-1})-w^{*+}_{\varepsilon,T}(1-q_{i})\right) v(\xi_i-x) \\
 &= & \sum_{i=-m}^{-1} \left(w^{*-}_{\varepsilon}(q_{i})-w^{*-}_{\varepsilon}(q_{i-1})\right) v(\xi_i-x)
 +\sum_{i=0}^n\left( w^{*+}_{\varepsilon}(1-q_{i-1})-w^{*+}_{\varepsilon}(1-q_{i})\right) v(\xi_i-x) \\
 &\geq& \sup_{(w^-,w^+)\in B(w_0,r)\times B(w_0,r)}\left\{\sum_{i=-m}^{-1}\left(w^{-} (q_{i})-w^{-}(q_{i-1})\right) v(\xi_i-x) \right.\\
&& \left. +\sum_{i=0}^n\left(w^{+}(1-q_{i-1})-w^{+}(1-q_{i})\right) v(\xi_i-x)\right\}-\varepsilon
=\phi(v)-\varepsilon,
 \edeq
 where the last inequality is due to (\ref{eq:w-sup-epsilon}).
This means $\widetilde{\phi}(v)\geq \phi(v)$ since $\varepsilon$ can be arbitrarily small.
This shows  $\phi(v)=\widetilde{\phi}(v)$ as desired.

\underline{\bf Step 3.}
Let $y_i :=\int_{t_i}^{t_{i+1}} g(t)dt$.
Since ${\mathscr G}$ constitutes all measurable functions with $|g(t)|\leq \tilde g(t)$ for $t\in[0,1]$,
 then we have $|y_i|\leq \int_{t_l}^{t_{l+1}} \tilde g(t)dt$.
The inequality in ${\widetilde \Psi}(\psi^0,r)$ (\ref{eq: set-psi-T-derivative-weighting}) can be reformulated as follows
\bgeqn
\sup\limits_{g\in {\mathscr G}} \left|\sum_{l=1}^T\int_{t_l}^{t_{l+1}}g(t)(\psi_l-\psi^0_{l})dt \right| 
&= \sup\limits_{y_1,\cdots,y_T} & \sum_{l=1}^T(\psi_l-\psi^0_{l})y_l \nonumber \\
&{\rm s.t.}&  |y_l|\leq \int_{t_l}^{t_{l+1}} \tilde g(t) dt,\, l=1,\cdots,T.
\edeqn
The right hand side  is a linear program. We derive its Lagrange dual. 
Let 
$\eta\in \R^{T}_+$ 
and $\theta\in \R^{T}_+$ 
denote the dual variables corresponding to 
constraints 
$-y_l\leq \int_{t_l}^{t_{l+1}} \tilde g(t) dt$ for $l=1,\cdots,T$
and 
$y_l\leq \int_{t_l}^{t_{l+1}} \tilde g(t) dt$ for $l=1,\cdots,T$ respectively.
Let
\bgeq
L(y;\eta,\theta)
&:=& \sum_{l=1}^T(\psi_l-\psi^0_{l})y_l 
- \sum_{l=1}^T \eta_l \left(-y_l -\int_{t_l}^{t_{l+1}} \tilde g(t) dt \right)
- \sum_{l=1}^T \theta_l \left(y_l-\int_{t_l}^{t_{l+1}} \tilde g(t) dt \right)\\
&=& \sum_{l=1}^T\left( \psi_l-\psi^0_{l}+\eta_l-\theta_l \right) y_l
+\sum_{l=1}^T\eta_l \int_{t_l}^{t_{l+1}} \tilde g(t) dt
+\sum_{l=1}^T \theta_l\int_{t_l}^{t_{l+1}} \tilde g(t) dt.
\edeq
The Lagrange dual can be written as
\bgeq
&\inf\limits_{\eta,\theta\geq 0} &\sum_{l=1}^T(\eta_l+\theta_l) \int_{t_l}^{t_{l+1}} \tilde g(t) dt\\
&{\rm s.t. }& \psi_l-\psi^0_{l}+\eta_l-\theta_l =0, l=1,\cdots,T.
\edeq
Thus $\phi(v)$ in (\ref{robust_W-BN-xi}) can be  reformulated as
\bgeqn
\label{eq:sub_inf_W}
\phi(v)&=&\sup_{(\psi^-,\psi^+)\in {\widetilde \Psi}(\psi^0,r)\times {\widetilde \Psi}(\psi^0,r)}
\left\{\sum_{i=-m}^{-1}v^{-}(\xi_i-x)\int_{q_{i-1}}^{q_{i}}\psi^-(t)dt
+\sum_{i=0}^n v^{+}(\xi_i-x)\int_{1-q_{i}}^{1-q_{i-1}}\psi^+(t)dt\right \} \nonumber \\
&=&\sup_{\tilde \psi^-,\tilde \psi^+, \eta,\theta}
\sum_{i=-m}^{-1} v^{-}(\xi_i-x)\sum_{l=1}^T\psi_l^-\int_{q_{i-1}}^{q_{i}}\mathds{1}_{[t_l,t_{l+1})}(t)dt 
 +\sum_{i=0}^{n} v^{+}(\xi_i-x) \sum_{l=1}^T\psi^+_l \int_{1-q_{i}}^{1-q_{i-1}}\mathds{1}_{[t_l,t_{l+1})}(t)dt \nonumber \\
 &&\quad {\rm s.t. } \quad (\ref{robust_Wra})-(\ref{robust_Wrj}),
\edeqn
where $\tilde \psi^-=(\psi^-_l,\cdots, \psi^-_T)$ and 
$\tilde \psi^+=(\psi^+_l,\cdots, \psi^+_T)$.

\underline{\bf Step 4.}
We are now ready to reformulate the constraint function in (\ref{eq:robust-V-W-computation-SAA}) 
$$
\sup_{(v,w^-,w^+)\in ({\cal V}_{El},{\cal W}_r,{\cal W}_r)} \sum_{i=-m}^n \pi_i v(\xi_i-x)
=\sup_{v\in {\cal V}_{El}} \phi(v).
$$
First, we will prove 
\bgeqn
\label{eq:phi-V2-V2widetildeM}
\sup_{v\in {\cal V}_{El}} \phi(v) 
=\sup_{v\in {\cal {\widetilde V}}_{El}} \phi(v),
\edeqn
where ${\cal {\widetilde V}}_{El}$ is defined as follows.

Let $\{y_j\}_{j=1}^{J}$ be the $j$-th entry of ranked support set 
$\mathscr S:= \bigcup_{m=1}^M( \supp(\gamma_m)\cup \supp(\eta_m))\cup\{a,b\}\cup \{0\}\cup \bigcup_{i=1}^N(\xi_i-x)\cup \bigcup_{k=1}^K\supp(A_k-a_k^{+})\cup \bigcup_{k=1}^K\supp(A_k-a_k^{-})$.
We will use $y_j$ to denote the $j$-th smallest entry of $\mathscr S$.
Let $j_{0}$ be the index satisfying $y_{j_{0}}=0$.
Denote 
\bgeq
\mathscr {\tilde V}:=\left\{ v:[\alpha,\beta]\rightarrow \R\;:
\begin{array}{ll}
&v \inmat{ is negative and convex on $[\alpha,0)$, and positive and concave on $(0,\beta]$}\\
& v \inmat{ is increasing continuous, } v(\alpha)=b_l, v(\beta)=b_r
\end{array}
\right\}.
\edeq
Let $\mathscr {\tilde V}_{J}$ denote the set of all piecewise linear functions $v\in \mathscr {\tilde V}$ with breakpoints 
$a=y_1<y_2<\cdots<y_{J}=b$,
$y_1,\cdots,y_J \in \mathscr S$.
For 
$$
{\cal V}_{El}={\cal V}^K_{ce}\cap {\cal V}^M_{pair}\cap {\cal V}_S=\left\{v\in {\mathscr {\tilde V}}: 
\begin{array}{ll}
&\mathbb E_{w^-w^+}[v(\gamma_m)]\leq \mathbb E_{w^-w^+}[v(\eta_m)], m=1,\cdots,M\\
& \bbe_{w^-w^+}[v(A_k-a_k^{+})]\leq 0, k=1,\cdots,K\\
& \bbe_{w^-w^+}[v(A_k-a_k^{-})]\geq 0,k=1,\cdots,K 
\end{array}
\right \},
$$
we define the ambiguity set
$$
{\cal {\widetilde V}}_{El}={\cal V}^K_{ce,J}\cap{\cal V}^M_{pair,J}\cap {\cal V}_S=
\left\{v_{J}\in \mathscr {\tilde V}_{J}: 
\begin{array}{ll}
&\mathbb E_{w^-w^+}[v_{J}(\gamma_m)]\leq \mathbb E_{w^-w^+}[v_{J}(\eta_m)], m=1,\cdots,M \\
& \bbe_{w^-w^+}[v_{J}(A_k-a_k^{+})]\leq 0,k=1,\cdots,K \\
& \bbe_{w^-w^+}[v_{J}(A_k-a_k^{-})]\geq 0,k=1,\cdots,K 
\end{array}
\right\}.
$$
We consider to approximate ${\cal V}_{El}$ by ${\cal {\widetilde V}}_{El}$.
The approximation function $v_J\in \mathscr {\tilde V}_J$ preserves convexity, concavity and increasing property,
and then $\mathscr {\tilde V}_J\subset \mathscr {\tilde V}$.
We choose $\gamma_m$ and $\eta_m$, $m=1,\cdots,M$ with discrete distributions,
and the breakpoints cover their 
support sets.
Since 
\bgeq
 \bbe_{w^-w^+}[v(\xi)]
  &=&\int_{0}^{\infty}w^+(\mathbb P(v(\xi)\geq t))dt
  -\int_{-\infty}^0 w^-(\mathbb P(v(\xi)\leq t))dt\\
 &=&\int_{0}^{\infty}w^+(1-F_{\xi}(t))dv^+(t)
 -\int_{-\infty}^0 w^-(F_{\xi}(t))dv^-(t),
\edeq
then
\bgeq
\mathbb E_{w^-w^+}[v_{J}(\gamma_m)]- \mathbb E_{w^-w^+}[v_{J}(\eta_m)] 
&=&\int_{y_{1}}^{y_{J}} \phi_m(y)dv_{J}(y) = \sum_{i=1}^{J-1}\phi_m(y_j)\Big(v_{J}(y_{i+1})-v_{J}(y_{i})\Big) \\
&=& \sum_{i=1}^{J-1}\phi_m(y_j)\Big(v(y_{i+1})-v(y_{i})\Big) 
=\mathbb E_{w^-w^+}[v(\gamma_m)]- \mathbb E_{w^-w^+}[v(\eta_k)],
\edeq
where $\phi_m$ is defined as in (\ref{eq:defi_phim}).
\bgeq
&& \bbe_{w^-w^+}[v_{J}(A_k-a_k^{+})]\\
&=&\int_{y_{j_0}}^{y_{J}} w^+(1-F_{A_k-a_k^{+}}(y))dv_{J}^+(y)
-\int_{y_1}^{y_{j_0}} w^-(F_{A_k-a_k^{+}}(y))dv_{J}^-(y)  \\
&=&\sum_{j={j_0}}^{J-1}w^+(1-F_{A_k-a_k^{+}}(y_j))\Big(v_{J}^+(y_{j+1})-v_{J}^+(y_{j})\Big)
 -\sum_{j=1}^{{j_0}-1}w^-(F_{A_k-a_k^{+}}(y_j)) \Big(v_{J}^-(y_{j+1})-v_{J}^-(y_{j})\Big)\\
&=&\mathbb E_{w^-w^+}[v(A_k-a_k^{+})]
\edeq
and
\bgeq
&& \bbe_{w^-w^+}[v_{J}(A_k-a_k^{-})]\\
&=&\int_{y_{j_0}}^{y_{J}} w^+(1-F_{A_k-a_k^{-}}(y))dv_{J}^+(y)
-\int_{y_1}^{y_{j_0}} w^-(F_{A_k-a_k^{-}}(y))dv_{J}^-(y)  \\
&=&\sum_{j={j_0}}^{J-1}w^+(1-F_{A_k-a_k^{-}}(y_j))\Big(v_{J}^+(y_{j+1})-v_{J}^+(y_{j})\Big)
 -\sum_{j=1}^{{j_0}-1}w^-(F_{A_k-a_k^{-}}(y_j)) \Big(v_{J}^-(y_{j+1})-v_{J}^-(y_{j})\Big)\\
&=&\mathbb E_{w^-w^+}[v(A_k-a_k^{-})].
\edeq
This shows ${\cal {\widetilde V}}_{El}\subset {\cal V}_{El}$,
and 
hence $\displaystyle  \sup_{v\in {\cal {\widetilde V}}_{El}} \phi(v)\leq \sup_{v\in {\cal V}_{El}} \phi(v)$.
Conversely,
let $\varepsilon>0$,
and $v_{\varepsilon}^*\in {\cal V}_{El}$
such that 
\bgeq
\phi(v_{\varepsilon}^*)
\geq 
\sup_{v\in {\cal V}_{El}} \phi(v)-\varepsilon.
\edeq
We can find a piecewise linear function $v_{\varepsilon,J}^{*}\in {\cal {\widetilde V}}_{El}$
such that $v^*_{\varepsilon,J} (\xi_i-x)= v^{*}_{\varepsilon}(\xi_i-x)$ for $1\leq i\leq N$.
Specifically,
for any $v\in {\cal V}_{El}$, we can construct
such a $v_{J}$ 
as
$$
v_{J}(y):=v(y_{i-1})+\frac{v(y_i)-v(y_{i-1})}{y_i-y_{i-1}}(y-y_i), \mbox{ for } y\in [y_{i-1},y_i], i=2,\cdots, J,
$$
where $y_1=a, y_{J}=b$. 
Observe that $v_{J}$ satisfies $v_{J}(y_i)=v(y_i)$ for each $i$.
Consequently
we have 
\bgeq
\sup_{v\in {\cal {\widetilde V}}_{El}} \phi(v)
\geq  \phi(v^*_{\varepsilon, J}) 
= \phi(v^*_{\varepsilon})
\geq \sup_{v\in {\cal V}_{El}} \phi(v)-\varepsilon.
\edeq
Then $\sup_{v\in {\cal {\widetilde V}}_{El}} \phi(v)\geq \sup_{v\in {\cal V}_{El}} \phi(v)$ follows as $\varepsilon$ can be arbitrarily small,
and thus (\ref{eq:phi-V2-V2widetildeM}) holds.

Next, by invoking (\ref{eq:sub_inf_W}),
we can reformulate the constraint function in (\ref{eq:robust-V-W-computation-SAA})  as 
\begin{subequations}
\begin{align}
&\sup_{(v,w^-,w^+)\in {\cal V}_{El}\times {\cal W}_r\times {\cal W}_r} \sum_{i=-m}^n \pi_i v(\xi_i-x) =\sup_{v\in {\cal V}_{El}} \phi(v) \nonumber \\
=&\sup_{v\in  {\cal {\widetilde V}}_{El}}\sup_{\tilde \psi^-,\tilde \psi^+, \eta,\theta}
\sum_{i=-m}^{-1}v^{-}(\xi_i-x)\sum_{l=1}^T\psi_l^-\int_{q_{i-1}}^{q_{i}}\mathds{1}_{[t_l,t_{l+1})}(t)dt 
+\sum_{i=0}^{n} v^{+}(\xi_i-x) \sum_{l=1}^T\psi^+_l\int_{1-q_{i}}^{1-q_{i-1}}\mathds{1}_{[t_l,t_{l+1})}(t)dt\nonumber \\
&\qquad \qquad   {\rm s.t.} \qquad 
  (\ref{robust_Wra})-(\ref{robust_Wrj}),\nonumber \\
=& \sup_{a,b} \sup_{\tilde \psi^-,\tilde \psi^+, \eta,\theta} \left\{\sum_{i=-m}^{-1} 
\Big(
(a_1(\xi_i-x)+b_1)\mathds{1}_{(-\infty,0)\cap[y_1,y_2]}(\xi_i-x) \right. \nonumber \\
& \qquad \qquad \qquad  +\sum_{j=2}^{J-1}(a_j(\xi_i-x)+b_j)\mathds{1}_{(-\infty,0)\cap(y_j,y_{j+1}]}(\xi_i-x) \Big) \sum_{l=1}^T\psi_l^- \int_{q_{i-1}}^{q_{i}}\mathds{1}_{[t_l,t_{l+1})}(t)dt \nonumber \\
& \qquad \qquad \qquad 
+\sum_{i=0}^{n} \Big(
 (a_1(\xi_i-x)+b_1)\mathds{1}_{(-\infty,0)\cap[y_1,y_2]}(\xi_i-x)\nonumber  \\
& \left. \qquad \qquad \qquad 
+\sum_{j=2}^{J-1}(a_j(\xi_i-x)+b_j)\mathds{1}_{(-\infty,0)\cap(y_j,y_{j+1}]}(\xi_i-x) \Big)
\sum_{l=1}^T\psi^+_l\int_{1-q_{i}}^{1-q_{i-1}}\mathds{1}_{[t_l,t_{l+1})}(t)dt
\right\}\nonumber \\
&\qquad \qquad   {\rm s.t.} \qquad 
  (\ref{robust_Wra})-(\ref{robust_Wrj}), ~~~ (\ref{robust_v_ob})-(\ref{robust_v_oj}).
\end{align}
\end{subequations}
The proof is complete.
\hfill $\Box$
\end{document}